\documentclass[a4paper]{amsart}
\usepackage[utf8]{inputenc}
\usepackage[T1]{fontenc}
\usepackage{lmodern}
\usepackage{amssymb}
\usepackage{nicefrac,mathtools,enumitem}
\usepackage{microtype}
\usepackage{amsxtra}

\usepackage{tikz}
\usetikzlibrary{matrix}
\tikzset{cd/.style=matrix of math nodes,row sep=2em,column sep=2em, text height=1.5ex, text depth=0.5ex}
\tikzset{cdar/.style=->,auto}

\setlist[enumerate,1]{label=\textup{(\arabic*)}}
\setlist[enumerate,2]{label=\textup{(\alph*)}}

\newcommand{\tenscorep}{\mathbin{\begin{tikzpicture}[baseline,x=.75ex,y=.75ex] \draw[line width=.2pt] (-0.8,1.15)--(0.8,1.15);\draw[line width=.2pt](0,-0.25)--(0,1.15); \draw[line width=.2pt] (0,0.75) circle [radius = 1];\end{tikzpicture}}}

\usepackage{xcolor}

\usepackage[pdftitle={Quantum group-twisted tensor products of C*-algebras II},
pdfauthor={Ralf Meyer, Sutanu Roy and Stanisław Lech Woronowicz},
pdfsubject={Mathematics; MSC 46L55 (81R50, 46L06, 18D10)}
]{hyperref}
\usepackage[lite]{amsrefs}
\newcommand*{\MRref}[2]{ \href{http://www.ams.org/mathscinet-getitem?mr=#1}{MR #1}}
\newcommand*{\arxiv}[1]{ \href{http://www.arxiv.org/abs/#1}{arXiv:#1}}
\renewcommand{\PrintDOI}[1]{\href{http://dx.doi.org/#1}{DOI #1}%
  \IfEmptyBibField{volume}{, (to appear in print)}{}}

\BibSpec{book}{%
  +{}  {\PrintPrimary}                {transition}
  +{,} { \textit}                     {title}
  +{.} { }                            {part}
  +{:} { \textit}                     {subtitle}
  +{,} { \PrintEdition}               {edition}
  +{}  { \PrintEditorsB}              {editor}
  +{,} { \PrintTranslatorsC}          {translator}
  +{,} { \PrintContributions}         {contribution}
  +{,} { }                            {series}
  +{,} { \voltext}                    {volume}
  +{,} { }                            {publisher}
  +{,} { }                            {organization}
  +{,} { }                            {address}
  +{,} { \PrintDateB}                 {date}
  +{,} { }                            {status}
  +{}  { \parenthesize}               {language}
  +{}  { \PrintTranslation}           {translation}
  +{;} { \PrintReprint}               {reprint}
  +{.} { }                            {note}
  +{.} {}                             {transition}
  +{,} { \PrintDOI}                   {doi}
  +{} { available at \url}            {eprint}
  +{}  {\SentenceSpace \PrintReviews} {review}
}

\numberwithin{equation}{section}

\theoremstyle{plain}
\newtheorem{theorem}[equation]{Theorem}
\newtheorem{lemma}[equation]{Lemma}
\newtheorem{proposition}[equation]{Proposition}

\newtheorem{corollary}[equation]{Corollary}

\theoremstyle{definition}
\newtheorem{definition}[equation]{Definition}

\theoremstyle{remark}
\newtheorem{remark}[equation]{Remark}

\newtheorem{example}[equation]{Example}

\newcommand*{\C}{\mathbb C}% Complex numbers
\newcommand*{\Z}{\mathbb Z}% Integers
\newcommand*{\blank}{\textup{\textvisiblespace}}%for arguments in functors
\newcommand*{\Un}{\textup{U}}%unitary group
\newcommand*{\diff}{\textup d}%used for dx in integrals etc.
\newcommand*{\univ}{\textup u}%universal
\newcommand*{\Ad}{\textup{Ad}}%conjugation by unitary
\newcommand*{\copp}{\textup{cop}}%co-opposite
\newcommand*{\defeq}{\mathrel{\vcentcolon=}}%defining condition
\newcommand*{\nb}{\nobreakdash}
\newcommand*{\Star}{*-}

\newcommand*{\Hils}[1][H]{\mathcal #1}%Hilbert space
\newcommand*{\Bound}{\mathbb B}%adjointable operators on a Hilbert module
\newcommand*{\Comp}{\mathbb K}%compact operators on a Hilbert module

\DeclareMathOperator{\Hom}{Hom}
\DeclareMathOperator{\Aut}{Aut}

\newcommand*{\Cont}{\textup{C}}%continuous functions
\newcommand*{\Contvin}{\textup C_0}%continuous functions vanishing at infinity
\newcommand*{\Cst}{\textup C^*}%C*-algebra
\newcommand*{\Cred}{\textup C^*_\textup r}%reduced group C*-algebra
\newcommand*{\CLS}{\mathrm{CLS}}%closed linear span
\newcommand*{\Mult}{\mathcal M}%multiplier algebra
\newcommand*{\U}{\mathcal U}%unitary group

\newcommand*{\Mor}{\textup{Mor}}%nondegenerate *-homomorphisms of C*-algebras
\newcommand*{\Id}{\textup{id}}%identity map
\newcommand*{\Cstcat}{\mathfrak{C^*alg}}%category of C*-algebras
\newcommand*{\Forget}{\mathsf{For}}%forgetful functor

\newcommand*{\Flip}{\Sigma}% flip operator on Hilbert space
\newcommand*{\flip}{\sigma}% flip map on the multiplier algebra

\newcommand*{\Bialg}[1]{(#1,\Comult[#1])}%C*-bialgebra
\newcommand*{\DuBialg}[1]{(\hat{#1},\DuComult[#1])}%dual C*-bialgebra

\newcommand*{\G}[1][G]{\mathbb #1}% quantum group
\newcommand*{\DuG}[1][G]{\widehat{\mathbb{#1}}}%dual quantum group

\newcommand*{\Qgrp}[2]{\mathbb{#1}=\Bialg{#2}}%quantum group as pair
\newcommand*{\DuQgrp}[2]{\widehat{\mathbb{#1}}=\DuBialg{#2}}%dual quantum group as pair

\newcommand*{\Comult}[1][]{\Delta_{#1}}%comultiplication
\newcommand*{\DuComult}[1][]{\hat{\Delta}_{#1}}%dualcomultiplication

\DeclareRobustCommand{\rchi}{{\mathpalette\irchi\relax}}
\newcommand*{\irchi}[2]{\raisebox{\depth}{$#1\chi$}} % inner command, used by \rchi

\newcommand*{\bichar}{\rchi}%bicharacter viewed as an element of the unitary multiplier
\newcommand*{\Dubichar}{\hat{\bichar}}%Dual bicharacter
\newcommand*{\Multunit}{\mathbb{W}}%muliplicative unitary acting on a hilbert space
\newcommand*{\multunit}[1][]{\textup{W}^{#1}}%multiplicative unitary as a bicharacter of the multiplier algebra
\newcommand*{\DuMultunit}{\widehat{\mathbb W}}%muliplicative unitary on action on a hilbert space
\newcommand*{\Dumultunit}[1][]{\widehat{\textup W}{}^{#1}}%multiplicative unitary as a bicharacter of the multiplier algebra

\newcommand*{\Rmattxt}{R}%R-matrix as text
\newcommand*{\Rmat}{\textup R}% R-matrix
\newcommand*{\UnivRmat}{\Rmat^\univ}% universal R-matrix
\newcommand*{\Rmatbraid}[2]{\textup X^{({#1},{#2})}}%rmatrix from braiding

\newcommand*{\Braiding}[2]{\begin{tikzpicture}[baseline]
    \draw[-] (0,0) -- (1.4ex,1.4ex) node[right,inner sep=0pt] {$\scriptstyle #2$};
    \draw[-,draw=white,line width=2.4pt] (0,1.4ex) -- (1.4ex,0);
    \draw[-] (1.4ex,0) -- (0,1.4ex) node[left,inner sep=0pt] {$\scriptstyle #1$};
  \end{tikzpicture}}
\newcommand*{\Dualbraiding}[2]{\begin{tikzpicture}[baseline]
    \draw[-] (1.4ex,0) -- (0,1.4ex) node[left,inner sep=0pt] {$\scriptstyle #1$};
    \draw[-,draw=white,line width=2.4pt] (0,0) -- (1.4ex,1.4ex);
    \draw[-] (0,0) -- (1.4ex,1.4ex) node[right,inner sep=0pt] {$\scriptstyle #2$};
  \end{tikzpicture}}

\newcommand*{\YDcat}{\mathcal{YD}\mathfrak{C^*alg}}%Category of Yetter--Drinfeld C*-algebras

\newcommand*{\corep}[1]{\textup{#1}}          %Corepresentation of a quantum group
\newcommand*{\Corep}[1]{\mathbb{#1}}          %Corepresentation as operator on Hilbert space
\newcommand*{\Ducorep}[1]{\hat{\corep{#1}}}   %Corepresentation viewed as dual (bicharacter sense)
\newcommand*{\DuCorep}[1]{\hat{\Corep{#1}}}   %Dual corepresentation viewed as operator on Hilbert space

\newcommand*{\maxcorep}[1][]{\mathcal{V}^{#1}}%maximal corepresentation of G-hat
\newcommand*{\dumaxcorep}[1][]{\tilde{\mathcal{V}}^{#1}}%maximal corepresentation of G

\newcommand*{\Corepcat}{\mathfrak{Corep}}%Category of corepresentations

\newcommand*{\Drinfdouble}[1]{\mathfrak{D}(#1)}%Drinfeld double
\newcommand*{\Codouble}[1]{\mathfrak{D}({#1}){ }\sphat\text{\space}}% Dual of Drinfeld double
\newcommand*{\DrinfDoubAlg}{\mathcal{D}}%Drinfeld double algebra
\newcommand*{\CodoubAlg}{\hat{\mathcal{D}}}%Quantum codouble algebra

\begin{document}
\title{Quantum group-twisted tensor products of \(\Cst\)-algebras II}

\author{Ralf Meyer}
\email{rmeyer2@uni-goettingen.de}
\address{Mathematisches Institut\\
  Georg-August Universität Göttingen\\
  Bunsenstraße 3--5\\
  37073 Göttingen\\
  Germany}
  
\author{Sutanu Roy}
\email{sr26@uottawa.ca}
\address{Department of Mathematics and Statistics\\
  University of Ottawa\\
  585 King Edward Avenue\\
  Ottawa, ON K1N 6N5\\
  Canada}

\author{Stanisław Lech Woronowicz}
\email{Stanislaw.Woronowicz@fuw.edu.pl}
\address{Instytut Matematyki\\Uniwersytet w Białymstoku, and\\
  Katedra Metod Matematycznych Fizyki\\
  Wydział Fizyki, Uniwersytet Warszawski\\
  Pasteura 5, 02-093 Warszawa, Poland}

\begin{abstract}
  For a quasitriangular \(\Cst\)\nb-quantum group, we enrich the
  twisted tensor product constructed in the first part of this series
  to a monoidal structure on the category of its continuous coactions
  on \(\Cst\)\nb-algebras.  We define braided \(\Cst\)\nb-quantum
  groups, where the comultiplication takes values in a twisted tensor
  product.  We show that compact braided \(\Cst\)\nb-quantum groups
  yield compact quantum groups by a semidirect product construction.
\end{abstract}


\thanks{Supported by the German Research Foundation (Deutsche Forschungsgemeinschaft (DFG)) through the Research Training Group 1493 and by the Alexander von Humboldt-Stiftung.}

\maketitle

\section{Introduction}
\label{sec:introduction}

Let \(C\) and~\(D\) be \(\Cst\)\nb-algebras with a coaction of a
\(\Cst\)\nb-quantum group \(\G=(A,\Comult[A])\).  As
in~\cite{Meyer-Roy-Woronowicz:Twisted_tensor}, \(\Cst\)\nb-quantum
groups are generated by manageable multiplicative unitaries, and Haar
weights are not assumed.  If~\(\G\) is a group, then the
\(\Cst\)\nb-tensor product \(C\otimes D\) inherits a diagonal
coaction.  This fails for quantum groups because the diagonal
coaction is not compatible with the multiplication in the tensor
product.  We use the noncommutative tensor products described
in~\cite{Meyer-Roy-Woronowicz:Twisted_tensor} to construct a
monoidal structure on the category of \(\G\)\nb-\(\Cst\)-algebras
if~\(\G\) is quasitriangular in a suitable sense.

Such a structure is to be expected from the analogous situation for
(co)module algebras over a Hopf algebra.  In that context, an
\(\Rmat\)\nb-matrix for the dual Hopf algebra allows to deform the
multiplication on the tensor product of two \(H\)\nb-comodule algebras
so as to get an \(H\)\nb-comodule algebra again.  For
\(\Cst\)\nb-quantum groups, Hopf module structures are replaced by
comodule structures.  Hence we call~\(\G\) quasitriangular if there is
a unitary \(\Rmat\)\nb-matrix \(\Rmat\in\U(\hat{A}\otimes\hat{A})\)
for the dual \(\Cst\)\nb-quantum group.

Since~\(\Rmat\) is a bicharacter, the braided tensor product
\(C\boxtimes D \defeq (C,\gamma)\boxtimes_\Rmat (D,\delta)\)
in~\cite{Meyer-Roy-Woronowicz:Twisted_tensor} is defined if
\((C,\gamma)\) and~\((D,\delta)\) are \(\Cst\)\nb-algebras with
continuous coactions of~\(\G\).  We show that \(C\boxtimes D\) carries
a unique continuous coaction \(\gamma\bowtie\delta\) of~\(\G\) for
which the canonical embeddings of \(C\) and~\(D\) are
\(\G\)\nb-equivariant.  (We do not denote this coaction
by~\(\gamma\boxtimes\delta\) because~\(\boxtimes\) is a bifunctor, and
the \Star{}homomorphism~\(\gamma\boxtimes\delta\) given by this
bifunctoriality is \emph{not} \(\gamma\bowtie\delta\).)

It is crucial for the theory here and
in~\cite{Meyer-Roy-Woronowicz:Twisted_tensor} that \(\Rmat\) is
unitary.  This rules out some important examples of quasitriangular
Hopf algebras.  For instance, \(\Rmat\)\nb-matrices for quantum
deformations of compact simple Lie groups are non-unitary.

If \((E,\epsilon)\) is another \(\Cst\)\nb-algebra with a continuous
coaction of~\(\G\), then there is a canonical isomorphism
\((C\boxtimes D)\boxtimes E\cong C\boxtimes (D\boxtimes E)\).
If \(C\) or~\(D\) carries a trivial \(\G\)\nb-coaction, then
\(C\boxtimes D=C\otimes D\), and \(\gamma\bowtie\delta\) is the
obvious induced action, \(\gamma\otimes\Id_D\) or
\(\Id_C\otimes\delta\).  Thus our tensor product on
\(\G\)\nb-coactions is monoidal: the tensor unit is~\(\C\) with
trivial coaction.  The tensor product of coactions is braided
monoidal if and only if it is symmetric monoidal, if and only if the
\(\Rmat\)\nb-matrix is antisymmetric.  This rarely happens, and it
should not be expected because this also usually fails on the Hopf
algebra level.  What should be braided is the category of Hilbert
space corepresentations.  This is indeed the case, and we use it to
prove that the tensor product for coactions is associative and
monoidal.

An \(\Rmat\)\nb-matrix \(\Rmat\in\U(\hat{A}\otimes\hat{A})\) lifts
uniquely to a \emph{universal} \(\Rmat\)\nb-matrix
\(\Rmat\in\U(\hat{A}^\univ\otimes\hat{A}^\univ)\) for the universal
quantum group~\(\hat{A}^\univ\), so it makes no difference whether
we consider \(\Rmat\)\nb-matrices for \(\hat{A}\)
or~\(\hat{A}^\univ\).  Since Hilbert space corepresentations
of~\(A\) are equivalent to Hilbert space representations
of~\(\hat{A}^\univ\), an \(\Rmat\)\nb-matrix for~\(\hat{A}^\univ\)
induces a braiding on the monoidal category of Hilbert space
corepresentations of~\(\G\).

If~\(\G\) is the quantum group of functions on an Abelian locally
compact group~\(\Gamma\), then its \(\Rmat\)\nb-matrices are simply
bicharacters \(\hat\Gamma\times\hat\Gamma\to \textup{U}(1)\).  For
instance, if~\(\Gamma\) is~\(\Z/2\), there are two such
bicharacters.  One gives the ordinary commutative tensor product
with the diagonal coaction, the other gives the skew-commutative
tensor product with diagonal \(\Z/2\)\nb-coaction.

A well-known class of quasitriangular Hopf algebras are Drinfeld
doubles; their module algebras are the same as Yetter--Drinfeld
algebras.  The dual of the Drinfeld double \(\Drinfdouble{\G}\)
of~\(\G\) is the Drinfeld codouble~\(\Codouble{\G}\), which we just
call quantum codouble.  (What we call quantum codouble is called
quantum double in~\cite{Nest-Voigt:Poincare}.)

For our class of \(\Cst\)\nb-quantum groups, quantum codoubles and
doubles and corresponding multiplicative unitaries are described
in~\cite{Roy:Codoubles}.  It is already shown in~\cite{Roy:Codoubles}
that quantum codoubles are quasitriangular and that
\(\Codouble{\G}\)-\(\Cst\)\nb-algebras are the same as
\(\G\)\nb-Yetter--Drinfeld \(\Cst\)\nb-algebras.  In this article, we
identify the twisted tensor product for the canonical
\(\Rmat\)\nb-matrix of~\(\Codouble{\G}\) with the twisted tensor
product in the category of \(\G\)-Yetter--Drinfeld
\(\Cst\)\nb-algebras constructed in~\cite{Nest-Voigt:Poincare}.

A \emph{braided \(\Cst\)\nb-bialgebra} is a
\(\G\)\nb-\(\Cst\)-algebra \((B,\beta)\) with a
comultiplication
\[
\Comult[B]\colon B\to B\boxtimes B
\]
that is coassociative.  We call \((B,\beta,\Comult[B])\) a
\emph{braided compact quantum group} if~\(B\) is unital
and~\(\Comult[B]\) satisfies an appropriate Podle\'s condition.

We are particularly interested in braided quantum groups over a
codouble~\(\Codouble{\G}\) because they appear in a
quantum group version of semidirect products.  In the group case,
the construction of a semidirect product~\(G\ltimes H\) requires a
conjugation action of~\(G\) on~\(H\) such that the multiplication
map \(H\times H\to H\) is \(G\)\nb-equivariant.  For quantum group
semidirect products, the equivariance of the multiplication map on
the underlying \(\Cst\)\nb-algebra~\(B\) of~\(\G[H]\) only makes
sense if we deform the tensor product because there is no canonical
\(\G\)\nb-coaction on \(B\otimes B\).  A theorem of
Radford~\cite{Radford:Hopf_projection}*{Theorem 3} for the analogous
situation in the world of Hopf algebras suggests that~\(\G[H]\)
should be a braided quantum group over the codouble
\(\Codouble{\G}\) of~\(\G\).  In this case, we describe an induced
\(\Cst\)\nb-bialgebra structure on~\(A\boxtimes B\).  This is a
\(\Cst\)\nb-algebraic analogue of what Majid calls ``bosonisation''
in~\cite{Majid:Hopfalg_in_BrdCat}.  We prefer to call the
construction of \(A\boxtimes B\) a ``semidirect product.''  If~\(A\)
is a compact quantum group and~\(B\) is a braided compact quantum
group, then their semidirect product \(A\boxtimes B\) is a compact
quantum group.

As a first example, we construct a \(\Cst\)\nb-algebraic analogue of
the partial duals studied
in~\cite{Barvels-Lentner-Schweigert:Partial_dual}.  We have
constructed braided quantum SU(2) groups with complex deformation
parameter~\(q\) together with Paweł Kasprzak
in~\cite{Kasprzak-Meyer-Roy-Woronowicz:Braided_SU2}; their
semidirect products are the deformation quantisations of the unitary
group U(2) defined in~\cite{Zhang-Zhao:Uq2}.

The construction of the \(\Cst\)\nb-bialgebra \(A\boxtimes B\) works
in great generality.  For this to be a \(\Cst\)\nb-quantum group, we
would need a multiplicative unitary for it.  Then it is best to work
on the level of multiplicative unitaries throughout.  That is done
in~\cite{Roy:Qgrp_with_proj}.  On that level, one can also go back
and decompose a semidirect product into the two factors.  Here we
limit our attention to the compact case, where bisimplifiability is
enough to get a quantum group.

We briefly summarise the following sections.  In
Section~\ref{sec:R-matrices}, we define \Rmattxt\nb-matrices and
show that they lift to the universal quantum group.  In
Section~\ref{sec:corepcat_qasi_triag_qntgrp}, we describe the
braided monoidal structure on the category of Hilbert space
corepresentations for a quasitriangular \(\Cst\)\nb-quantum group.
As an example, we consider the case of Abelian groups.  In
Section~\ref{sec:Braided_tens_Cst} we construct the ``diagonal''
action of a quasitriangular quantum group on tensor products twisted
by the \(\Rmat\)\nb-matrix and show that it gives a monoidal
structure on \(\G\)\nb-\(\Cst\)-algebras.
Section~\ref{sec:coact_cat_codoub} studies the case of quantum
codoubles, where coactions on \(\Cst\)\nb-algebras are equivalent to
Yetter--Drinfeld algebra structures.
Section~\ref{sec:braid_C_star_bialg} contains the semidirect product
construction for braided \(\Cst\)\nb-bialgebras.  The appendix
recalls basic results about \(\Cst\)\nb-quantum groups and some
results of our previous articles for the convenience of the reader.
There are also some new observations about Heisenberg pairs in
Appendix~\ref{subsec:Heis_pair_crossed_prod}, which would fit better
into~\cite{Meyer-Roy-Woronowicz:Twisted_tensor} but were left out
there.

\section{R-matrices}
\label{sec:R-matrices}

Let \(\Qgrp{G}{A}\) be a \(\Cst\)\nb-quantum group and let
\(\multunit\in\U(\hat{A}\otimes A)\) be its reduced bicharacter; see
Appendix~\ref{sec:multunit_quantum_groups} and
Definition~\ref{def:bicharacter}.

\begin{definition}
  \label{def:R_matix}
  A bicharacter \(\Rmat\in\U(A\otimes A)\) is called an
  \emph{\Rmattxt\nb-matrix} if
  \begin{equation}
    \label{eq:reduced_R_mat_equivariant_cond}
    \Rmat(\flip\circ\Comult[A](a))\Rmat^* = \Comult[A](a)
    \qquad\text{for all }a\in A.
  \end{equation}
\end{definition}

\begin{lemma}
  \label{lemm:dual_R_mat}
  The \emph{dual} \(\hat{\Rmat}\defeq \flip(\Rmat^*) \in
  \U(A\otimes A)\) of a bicharacter \(\Rmat\in\U(A \otimes A)\) is an
  \Rmattxt\nb-matrix if and only if~\(\Rmat\) is an
  \Rmattxt\nb-matrix.\qed
\end{lemma}

\begin{remark}
  \label{rem:R-mat}
  The standard convention for Hopf algebras assumes
  \(\Rmat(\Comult[A](a))\Rmat^* = \flip\circ\Comult[A](a)\), which is
  opposite to~\eqref{eq:reduced_R_mat_equivariant_cond}.  Our
  convention in~\ref{eq:reduced_R_mat_equivariant_cond} becomes the
  standard one (see \cite{Majid:Quantum_grp}*{Definition 2.1.1}) if we
  replace~\(\Comult[A]\) by \(\Comult[A]^\copp\defeq
  \flip\circ\Comult[A]\) or \(\Rmat\) by~\(\Rmat^*\).
\end{remark}

In order to simplify proofs later, we lift an \Rmattxt\nb-matrix
\(\Rmat\in\U(A\otimes A)\) to \(\U(A^\univ\otimes A^\univ)\):

\begin{proposition}
  \label{prop:univ_R_matrix}
  There is a unique \(\UnivRmat\in\U(A^\univ\otimes A^\univ)\) with
  \begin{alignat}{2}
    \label{eq:Univ_R_mat_and_red_R_mat}
    (\Lambda\otimes\Lambda)\UnivRmat &=\Rmat
    &\qquad& \text{in }\U(A\otimes A),\\
    \label{eq:Univ_R_matrix_char_in_first_leg}
    (\Comult[A^\univ]\otimes\Id_{A^\univ}) \UnivRmat &= \UnivRmat_{23} \UnivRmat_{13}
    &\qquad& \text{in }\U(A^\univ\otimes A^\univ\otimes A^\univ),\\
    \label{eq:Univ_R_matrix_char_in_second_leg}
    (\Id_{A^\univ}\otimes \Comult[A^\univ]) \UnivRmat &= \UnivRmat_{12} \UnivRmat_{13}
    &\qquad& \text{in }\U(A^\univ\otimes A^\univ\otimes A^\univ).
  \end{alignat}
  This unitary also satisfies
  \begin{align}
    \label{eq:Univ_R_matrix_equivariant_cond}
    \UnivRmat(\flip\circ\Comult[A]^\univ(a)) (\UnivRmat)^* &=
    \Comult[A]^\univ (a)
    \qquad \text{for all }a\in A^\univ.
  \end{align}
\end{proposition}

\begin{proof}
  \cite{Meyer-Roy-Woronowicz:Homomorphisms}*{Proposition~4.7} gives a
  unique \(\UnivRmat\in\U(A^\univ\otimes A^\univ)\) satisfying
  \eqref{eq:Univ_R_mat_and_red_R_mat}--\eqref{eq:Univ_R_matrix_char_in_second_leg}.
  The nontrivial part is to show that~\(\UnivRmat\) satisfies
  \eqref{eq:Univ_R_matrix_equivariant_cond}.  Let
  \(\maxcorep\in\U(\hat{A}\otimes A^\univ)\) be the universal
  bicharacter as in
  Appendix~\ref{sec:univ_qgr}.  Theorem~25 and Proposition~31
  in~\cite{Soltan-Woronowicz:Multiplicative_unitaries} show that
  \begin{equation}
    \label{eq:slice_of_maxcorep_and_corep_cond}
    A^\univ=\{(\omega\otimes\Id_{A^\univ})\maxcorep \mid \omega\in \hat{A}'\}^{\CLS}
    \qquad\text{and}\qquad
    (\Id_{\hat{A}}\otimes\Comult[A^\univ])\maxcorep=\maxcorep_{12}\maxcorep_{13}.
  \end{equation}
  Therefore, \eqref{eq:Univ_R_matrix_equivariant_cond} is equivalent to:
  \begin{equation}
    \label{eq:commutation_rel_univ_rmat_and_maxcorep_in_terms_of_comult}
    \UnivRmat_{23} \maxcorep_{13}\maxcorep_{12} (\UnivRmat_{23})^*
    =\maxcorep_{12}\maxcorep_{13}
    \qquad \text{in }\U(\hat{A}\otimes A^\univ\otimes A^\univ).
  \end{equation}
  The unitary \(\overline{\Rmat} \defeq
  (\Lambda\otimes\Id_{A^\univ})\UnivRmat \in\Mult(A\otimes A^\univ)\)
  is also a bicharacter.  Let \(X \defeq
  \multunit[*]_{12}\overline{\Rmat}_{23}\maxcorep_{13}\multunit_{12} \in
  \U(\hat{A}\otimes A\otimes A^\univ)\).  The following computation
  shows that~\(X\) is a corepresentation of~\((A^\univ,\Comult[A^\univ])\) on 
  \(\hat{A}\otimes A\):
  \begin{multline*}
    (\Id_{\hat{A}}\otimes\Id_A\otimes\Comult[A^\univ])
    (\multunit[*]_{12}\overline{\Rmat}_{23}\maxcorep_{13}\multunit_{12})
    = \multunit[*]_{12}\overline{\Rmat}_{23}\overline{\Rmat}_{24}
    \maxcorep_{13}\maxcorep_{14}\multunit_{12}\\
    = (\multunit[*]_{12}\overline{\Rmat}_{23}\maxcorep_{13}\multunit_{12})
    (\multunit[*]_{12}\overline{\Rmat}_{24}\maxcorep_{14}\multunit_{12})
    = X_{123}X_{124}.
  \end{multline*}
  The first step uses~\eqref{eq:Univ_R_matrix_char_in_second_leg} and
  \eqref{eq:slice_of_maxcorep_and_corep_cond}, the second step uses
  that \(\overline{\Rmat}_{24}\) and~\(\maxcorep_{13}\) commute, and
  the last step is trivial.  A similar routine computation shows that
  \(Y \defeq \maxcorep_{13}\overline{\Rmat}_{23}\in \U(\hat{A}\otimes
  A\otimes A^\univ)\) satisfies
  \((\Id_{\hat{A}}\otimes\Id_A\otimes\Comult[A^\univ])Y =
  Y_{123}Y_{124}\).

  The argument that shows
  that~\eqref{eq:Univ_R_matrix_equivariant_cond} is equivalent
  to~\eqref{eq:commutation_rel_univ_rmat_and_maxcorep_in_terms_of_comult}
  also shows that the \Rmattxt\nb-matrix
  condition~\eqref{eq:reduced_R_mat_equivariant_cond} is equivalent to
  \begin{equation}
    \label{eq:equiv_red_R_mat_equivariant_cond}
    \Rmat_{23}\multunit_{13}\multunit_{12} =\multunit_{12}\multunit_{13}\Rmat_{23}
    \qquad\text{in }\U(\hat{A}\otimes A\otimes A).
  \end{equation}
  Thus \((\Id_{\hat{A}}\otimes\Id_A\otimes\Lambda)X =
  (\Id_{\hat{A}}\otimes\Id_A\otimes\Lambda)Y\).  Now we use Lemma
  \cite{Meyer-Roy-Woronowicz:Homomorphisms}*{Lemma 4.6}, which is a
  variation on \cite{Kustermans:LCQG_universal}*{Result 6.1}.  It
  gives \(X=Y\) or, equivalently,
  \begin{equation}
    \label{eq:equality_of_coreps_in_third_leg}
    \maxcorep[*]_{13} \multunit[*]_{12}\bar{\Rmat}_{23}\maxcorep_{13}
    = \bar{\Rmat}_{23} \multunit[*]_{12}
    \qquad \text{in }\U(\hat{A}\otimes A\otimes A^\univ).
  \end{equation}
  Similarly, \(\widetilde{X}\defeq \maxcorep[*]_{13}
  (\UnivRmat_{23})^* \maxcorep_{12} \maxcorep_{13}\) and
  \(\widetilde{Y}\defeq \maxcorep_{12} (\UnivRmat_{23})^*\) in
  \(\U(\hat{A}\otimes A^\univ\otimes A^\univ)\) satisfy
  \((\Id_{\hat{A}}\otimes\Lambda\otimes\Id_{A^\univ})\widetilde{X} =
  (\Id_{\hat{A}}\otimes\Lambda\otimes\Id_{A^\univ})\widetilde{Y}\)
  by~\eqref{eq:equality_of_coreps_in_third_leg}, and
  \[
  (\Id_{\hat{A}}\otimes\Comult[A^\univ]\otimes\Id_{A^\univ})(\widetilde{X})
  = \widetilde{X}_{124} \widetilde{X}_{134},\qquad
  (\Id_{\hat{A}}\otimes\Comult[A^\univ]\otimes\Id_{A^\univ})(\widetilde{Y})
  = \widetilde{Y}_{124} \widetilde{Y}_{134}
  \]
  because of \eqref{eq:Univ_R_matrix_char_in_first_leg}
  and~\eqref{eq:slice_of_maxcorep_and_corep_cond}.  Another
  application of \cite{Meyer-Roy-Woronowicz:Homomorphisms}*{Lemma 4.6}
  gives \(\widetilde{X}=\widetilde{Y}\), which is equivalent
  to~\eqref{eq:commutation_rel_univ_rmat_and_maxcorep_in_terms_of_comult}.
\end{proof}

\cite{Soltan-Woronowicz:Multiplicative_unitaries}*{Proposition~31.2}
shows that~\((A^\univ,\Comult[A^\univ])\) has a bounded counit: there
is a unique morphism \(e\colon A^\univ\to\C\) with
\begin{equation}
  \label{eq:bdd_counit_and_univ_comult}
  (e\otimes\Id_{A^\univ})\Comult[A^\univ]
  = (\Id_{A^\univ}\otimes e)\Comult[A^\univ]
  = \Id_{A^\univ}.
\end{equation}

\begin{lemma}
  \label{lemm:prop_of_univ_Rmat}
  The unitary~\(\UnivRmat\in\U(A^\univ\otimes A^\univ)\) in
  Proposition~\textup{\ref{prop:univ_R_matrix}} satisfies
  \begin{alignat}{2}
    \label{eq:Univ_R_matrix_and_counit}
    (e\otimes\Id_{A^\univ})\UnivRmat
    &= (\Id_{A^\univ}\otimes e)\UnivRmat = 1_{A^\univ}
    &\qquad&\text{in }\U(A^\univ),\\
    \label{eq:R_matrix_Yang_Baxter}
    \UnivRmat_{12}\UnivRmat_{13}\UnivRmat_{23} &=
    \UnivRmat_{23}\UnivRmat_{13}\UnivRmat_{12}
    &\qquad&\text{in }\U(A^\univ\otimes A^\univ\otimes A^\univ).
  \end{alignat}
\end{lemma}

\begin{proof}
  Apply \(\Id_{A^\univ}\otimes e\otimes\Id_{A^\univ}\) on both
  sides of~\eqref{eq:Univ_R_matrix_char_in_first_leg}
  and~\eqref{eq:Univ_R_matrix_char_in_second_leg} and then
  use~\eqref{eq:bdd_counit_and_univ_comult}.  This gives
  \[
  \UnivRmat
  = (1_{A^\univ}\otimes (e\otimes\Id_{A^\univ})\UnivRmat) \UnivRmat
  = (((\Id_{A^\univ}\otimes e) \UnivRmat) \otimes 1_{A^\univ})\UnivRmat.
  \]
  Multiplying with~\((\UnivRmat)^*\) on the right gives
  \((e\otimes\Id_{A^\univ})\UnivRmat = (\Id_{A^\univ}\otimes
  e)\UnivRmat=1_{A^\univ}\).

  The following computation yields~\eqref{eq:R_matrix_Yang_Baxter}:
  \[
  \UnivRmat_{12}\UnivRmat_{13}\UnivRmat_{23}
  = ((\Id_{A^\univ}\otimes\Comult^\univ)\UnivRmat)\UnivRmat_{23}
  = \UnivRmat_{23}((\Id_{A^\univ}\otimes\flip\circ\Comult^\univ)\UnivRmat)
  = \UnivRmat_{23} \UnivRmat_{13} \UnivRmat_{12};
  \]
  here the first and third step
  use~\eqref{eq:Univ_R_matrix_char_in_second_leg} and the second step
  uses~\eqref{eq:Univ_R_matrix_equivariant_cond}.
\end{proof}

\section{Corepresentation categories of quasitriangular quantum groups}
\label{sec:corepcat_qasi_triag_qntgrp}

\begin{definition}
  \label{def:quasi_triang}
  A \emph{quasitriangular} \(\Cst\)\nb-quantum group is a
  \(\Cst\)\nb-quantum group \(\Qgrp{G}{A}\) with an \Rmattxt\nb-matrix
  \(\Rmat\in\U(\hat{A}\otimes\hat{A})\).
\end{definition}

Let \(\Flip^{(\Hils_1,\Hils_2)}\colon
\Hils_1\otimes\Hils_2\to\Hils_2\otimes\Hils_1\) denote the flip
operator.  As already pointed out
in~\cite{Soltan-Woronowicz:Multiplicative_unitaries},
\(\Flip^{(\Hils_1,\Hils_2)}_{12}\) is \(\G\)\nb-equivariant for all
corepresentations of~\(\G\) if and only if~\(\G\) is commutative.
Hence~\(\Flip^{(\cdot,\cdot)}\) does not give a braiding
on~\(\Corepcat(\G)\) in general.

Let \(\corep{U}^{\Hils_i}\in\U(\Comp(\Hils_i)\otimes A)\) be
corepresentations of~\(\G\) on~\(\Hils_i\) for \(i=1,2\).  These
correspond to representations of the universal quantum
group~\(\hat{A}^\univ\) by the universal property
of~\(\hat{A}^\univ\).  More precisely, there are unique
\(\hat{\varphi}_i\in\Mor(\hat{A}^\univ,\Comp(\Hils_i))\) such that
\((\hat{\varphi}_i\otimes\Id_A)\dumaxcorep=\corep{U}^{\Hils_i}\) for
\(i=1,2\), see Appendix~\ref{sec:univ_qgr}; here~\(\dumaxcorep[A]\)
is the universal bicharacter in \(\U(\hat{A}^\univ\otimes A)\).

Define \(\Braiding{\Hils_1}{\Hils_2}
\colon\Hils_1\otimes\Hils_2\to\Hils_2\otimes\Hils_1\) by
\begin{alignat}{2}
  \label{eq:Yang_baxter_unit_corep}
  \Rmatbraid{\Hils_2}{\Hils_1} &\defeq
  (\hat{\varphi}_2\otimes\hat{\varphi}_1)(\UnivRmat)^*
  &\quad&\text{in }\U(\Hils_2\otimes\Hils_1),\\
  \label{eq:braiding_induced_by_R-matrix}
  \Braiding{\Hils_1}{\Hils_2} &\defeq
  \Rmatbraid{\Hils_2}{\Hils_1}\circ\Flip^{\Hils_1,\Hils_2}
  &\quad&\text{in }\U(\Hils_1\otimes\Hils_2,\Hils_2\otimes\Hils_1).
\end{alignat}
Here \(\UnivRmat\in\U(\hat{A}^\univ\otimes\hat{A}^\univ)\) is as in
Proposition~\ref{prop:univ_R_matrix}.

\begin{proposition}
  \label{prop:braiding_quasi_qgrp_and_R_matrix}
  The unitaries \(\Braiding{\Hils_1}{\Hils_2}\colon
  \Hils_1\otimes\Hils_2\to\Hils_2\otimes\Hils_1\) are
  \(\G\)\nb-equivariant, that is,
  \begin{align}
    \label{eq:braiding_equivariant}
    \Braiding{\Hils_1}{\Hils_2}_{12}(\corep{U}^{\Hils_1}\tenscorep \corep{U}^{\Hils_2})
    &=(\corep{U}^{\Hils_2}\tenscorep \corep{U}^{\Hils_1})\Braiding{\Hils_1}{\Hils_2}_{12}
    \qquad \text{in }\U(\Comp(\Hils_1\otimes\Hils_2)\otimes A)
  \end{align}
  for all \(\corep{U}^{\Hils_1},\corep{U}^{\Hils_2}\in\Corepcat(\G)\).
  The tensor product~\(\tenscorep\) is defined
  in~\eqref{eq:tens_corep}.

  The unitaries \(\Braiding{\Hils_1}{\Hils_2}\) define a braiding on
  \(\Corepcat(\G)\), that is, the following hexagons commute for all
  \(\corep{U}^{\Hils_i}\in\Corepcat(\G)\), \(i=1,2,3\):
  \begin{gather}
    \label{eq:first_braid_diag}
    \begin{tikzpicture}[scale=3,baseline=(current bounding box.west)]
      \node (m-1-2) at (.4,.5)
      {$\Hils_1\otimes(\Hils_2\otimes\Hils_3)$};
      \node (m-1-3) at (2.1,.5)
      {$(\Hils_2\otimes\Hils_3)\otimes\Hils_1$};
      \node (m-2-1) at (0,0)
      {$(\Hils_1\otimes\Hils_2)\otimes\Hils_3$};
      \node (m-2-4) at (2.5,0)
      {$\Hils_2\otimes(\Hils_3\otimes\Hils_1)$};
      \node (m-3-2) at (.4,-.5)
      {$(\Hils_2\otimes\Hils_1)\otimes\Hils_3$};
      \node (m-3-3) at (2.1,-.5)
      {$\Hils_2\otimes(\Hils_1\otimes\Hils_3)$};
      \draw[cdar] (m-2-1) -- (m-1-2);
      \draw[cdar] (m-1-2) -- node
      {$\Braiding{\Hils_1}{\Hils_2\otimes\Hils_3}$} (m-1-3);
      \draw[cdar] (m-1-3) -- (m-2-4);
      \draw[cdar] (m-2-1) -- node[swap]
      {$\Braiding{\Hils_1}{\Hils_2}\otimes\Id_{\Hils_3}$}
      (m-3-2);
      \draw[cdar] (m-3-2) --
      node[swap] {\phantom{$\Braiding{\Hils_1\otimes\Hils_2}{\Hils_3}$}}
      (m-3-3);
      \draw[cdar] (m-3-3) -- node[swap]
      {$\Id_{\Hils_2}\otimes \Braiding{\Hils_1}{\Hils_3}$}
      (m-2-4);
    \end{tikzpicture}
    \\
    \label{eq:second_braid_diag}
    \begin{tikzpicture}[scale=3,baseline=(current bounding box.west)]
      \node (m-1-2) at (.4,.5)
      {$(\Hils_1\otimes\Hils_2)\otimes\Hils_3$};
      \node (m-1-3) at (2.1,.5)
      {$\Hils_3\otimes(\Hils_1\otimes\Hils_2)$};
      \node (m-2-1) at (0,0)
      {$\Hils_1\otimes(\Hils_2\otimes\Hils_3)$};
      \node (m-2-4) at (2.5,0)
      {$(\Hils_3\otimes\Hils_1)\otimes\Hils_2$};
      \node (m-3-2) at (.4,-.5)
      {$\Hils_1\otimes(\Hils_3\otimes\Hils_2)$};
      \node (m-3-3) at (2.1,-.5)
      {$(\Hils_1\otimes\Hils_3)\otimes\Hils_2$};
      \draw[cdar] (m-2-1) -- (m-1-2);
      \draw[cdar] (m-1-2) -- node
      {$\Braiding{\Hils_1\otimes\Hils_2}{\Hils_3}$} (m-1-3);
      \draw[cdar] (m-1-3) -- (m-2-4);
      \draw[cdar] (m-2-1) -- node[swap]
      {$\Id_{\Hils_1}\otimes \Braiding{\Hils_2}{\Hils_3}$}
      (m-3-2);
      \draw[cdar] (m-3-2) --
      node[swap] {\phantom{$\Braiding{\Hils_1\otimes\Hils_2}{\Hils_3}$}}
      (m-3-3);
      \draw[cdar] (m-3-3) -- node[swap]
      {$\Braiding{\Hils_1}{\Hils_3}\otimes\Id_{\Hils_2}$}
      (m-2-4);
    \end{tikzpicture}
  \end{gather}
  Here the unlabelled arrows are the standard associators of Hilbert spaces.
\end{proposition}

\begin{proof}
  We have \((\DuComult[A^\univ]\otimes\Id_A)\dumaxcorep[A] =
  \dumaxcorep[A]_{23}\dumaxcorep[A]_{13}\) in
  \(\U(\hat{A}^\univ\otimes\hat{A}^\univ\otimes A)\)
  because~\(\dumaxcorep[A]\) is a character in the first leg.
  Therefore, the corepresentation \(\corep{U}^{\Hils_1}\tenscorep
  \corep{U}^{\Hils_2}\) corresponds to
  \((\hat{\phi}_1\otimes\hat{\phi}_2)\circ
  \sigma\circ\DuComult[A]^\univ\colon\hat{A}^\univ\to
  \Bound(\Hils_1\otimes\Hils_2)\) through the universal
  property~\eqref{eq:univ_prop_dumaxcprep} of~\(\dumaxcorep\):
  \begin{equation}
    \label{eq:univ_dual_rep_tens_corep}
    ((\hat{\phi}_1\otimes\hat{\phi}_2)\circ\sigma\circ
    \DuComult[A]^\univ\otimes\Id_A)\dumaxcorep
    = \corep{U}^{\Hils_1}\tenscorep \corep{U}^{\Hils_2}.
  \end{equation}
  The following computation yields~\eqref{eq:braiding_equivariant}:
  \begin{align*}
    \Braiding{\Hils_1}{\Hils_2}_{12}(\corep{U}^{\Hils_1}\tenscorep \corep{U}^{\Hils_2})
    &= (\hat{\phi}_2\otimes\hat{\phi}_1\otimes\Id_A)\bigl((\UnivRmat_{12})^*(\DuComult[A]^\univ\otimes\Id_A)\dumaxcorep\bigr)
    \Flip^{(\Hils_1,\Hils_2)}_{12}\\
    &= \bigl(((\hat{\phi}_2\otimes\hat{\phi}_1)\circ\flip\circ\DuComult[A]^\univ)\otimes\Id_A)\dumaxcorep\bigr)
    \Rmatbraid{\Hils_2}{\Hils_1}_{12}\Flip^{(\Hils_1,\Hils_2)}_{12}\\
    &= (\corep{U}^{\Hils_2}\tenscorep \corep{U}^{\Hils_1})\Braiding{\Hils_1}{\Hils_2}_{12}.
  \end{align*}
  The first equality uses~\eqref{eq:univ_dual_rep_tens_corep} and \eqref{eq:braiding_induced_by_R-matrix},
  the second equality follows from~\eqref{eq:Univ_R_matrix_equivariant_cond} and \eqref{eq:Yang_baxter_unit_corep},
  and the last equality uses~\eqref{eq:univ_dual_rep_tens_corep} and \eqref{eq:braiding_induced_by_R-matrix}.

  Equations \eqref{eq:braiding_induced_by_R-matrix}
  and~\eqref{eq:univ_dual_rep_tens_corep} imply
  \begin{multline}
    \label{eq:R-mat_braid_1_23}
    \Braiding{\Hils_1}{\Hils_2\otimes\Hils_3}
    \defeq\Rmatbraid{\Hils_2\otimes\Hils_3}{\Hils_1}\Flip^{(\Hils_1,\Hils_2\otimes\Hils_3)}\\
    =      \Bigl((\hat{\phi}_2\otimes\hat{\phi}_3\otimes\hat{\phi}_1)
    (\flip\circ\DuComult[A]^\univ\otimes\Id_{\hat{A}^\univ})(\UnivRmat)^*\Bigr)
    \circ \Flip^{(\Hils_1,\Hils_2\otimes\Hils_3)}.
  \end{multline}
  Now we check the first braiding diagram~\eqref{eq:first_braid_diag}:
  \begin{multline*}
    \Bigl((\hat{\phi}_2\otimes\hat{\phi}_3\otimes\hat{\phi}_1)
    (\flip\circ\DuComult[A]^\univ\otimes\Id_{\hat{A}^\univ})(\UnivRmat)^*\Bigr)
    \circ \Flip^{(\Hils_1,\Hils_2\otimes\Hils_3)}\\
    = \Bigl((\hat{\phi}_2\otimes\hat{\phi}_3\otimes\hat{\phi}_1)\bigl((\UnivRmat)^*_{23}(\UnivRmat)^*_{13}\bigr)\Bigr)
    \Flip^{(\Hils_1,\Hils_3)}_{23}\Flip^{(\Hils_1,\Hils_2)}_{12}\\
    = \Rmatbraid{\Hils_3}{\Hils_1}_{23}\Rmatbraid{\Hils_2}{\Hils_1}_{13}\Flip^{(\Hils_1,\Hils_3)}_{23}\Flip^{(\Hils_1,\Hils_2)}_{12}\\
    = \Rmatbraid{\Hils_3}{\Hils_1}_{23}\Flip^{(\Hils_1,\Hils_3)}_{23}\Rmatbraid{\Hils_2}{\Hils_1}_{12}\Flip^{(\Hils_1,\Hils_2)}_{12}
    = \Braiding{\Hils_3}{\Hils_1}_{23}\Braiding{\Hils_2}{\Hils_1}_{12};
  \end{multline*}
  here the first equality
  uses~\eqref{eq:Univ_R_matrix_char_in_first_leg}, the second equality
  uses~\eqref{eq:Yang_baxter_unit_corep}, the third equality uses
  properties of the flip operator~\(\Flip\), and the fourth equality
  follows from~\eqref{eq:braiding_induced_by_R-matrix}.

  A similar computation
  for~\(\Braiding{\Hils_1\otimes\Hils_2}{\Hils_3}\) yields the second
  braiding diagram~\eqref{eq:second_braid_diag}.
\end{proof}

\begin{corollary}
  \label{cor:Yang-Baxter_rmat_braid}
  If \(\C\) carries the trivial corepresentation of~\(\G\),
  then
  \[
  \Braiding{\C}{\Hils}\colon \C\otimes\Hils\to\Hils\otimes\C
  \qquad\text{and}\qquad
  \Braiding{\Hils}{\C}\colon
  \Hils\otimes\C\to\C\otimes\Hils
  \]
  are the canonical isomorphisms.  For any three
  corepresentations of~\(\G\),
  \begin{equation}
    \label{eq:braiding_coherence}
    \Braiding{\Hils_1}{\Hils_2}_{23}
    \Braiding{\Hils_1}{\Hils_3}_{12}
    \Braiding{\Hils_2}{\Hils_3}_{23}
    = \Braiding{\Hils_2}{\Hils_3}_{12}
    \Braiding{\Hils_1}{\Hils_3}_{23}
    \Braiding{\Hils_1}{\Hils_2}_{12}.
  \end{equation}
\end{corollary}

\begin{proof}
  These are general properties of braided monoidal categories, see
  \cite{Joyal-Street:Braided}*{Proposition 2.1}.  They also follow
  from \eqref{eq:Univ_R_matrix_and_counit},
  \eqref{eq:R_matrix_Yang_Baxter},
  and~\eqref{eq:Yang_baxter_unit_corep}.
\end{proof}

\begin{remark}
  \label{rem:Braiding_induced_by_dual_Rmat}
  The dual \(\hat{\Rmat}\defeq \flip(\Rmat^*)\) of an
  \Rmattxt\nb-matrix \(\Rmat\in\U(\hat{A}\otimes\hat{A})\) is again an
  \Rmattxt\nb-matrix by Lemma~\ref{lemm:dual_R_mat}.  A routine
  computation shows that the resulting braiding on~\(\Corepcat(\G)\)
  is the dual braiding, given by the braiding unitaries
  \[
  \Dualbraiding{\Hils_1}{\Hils_2} =
  \bigl(\Braiding{\Hils_2}{\Hils_1}\bigr)^*\colon
  \Hils_1\otimes\Hils_2\to \Hils_2\otimes\Hils_1.
  \]
\end{remark}

\subsection{Symmetric braidings}
\label{sec:symmetric}

\begin{definition}
  \label{def:antisymmetric_R-matrix}
  An \Rmattxt\nb-matrix \(\Rmat\in\U(A \otimes A)\) is called
  \emph{antisymmetric} if \(\Rmat^*=\sigma(\Rmat)\) for the flip
  \(\sigma\colon A\otimes A\to A\otimes A\), \(a_1\otimes a_2\mapsto
  a_2\otimes a_1\).
\end{definition}

\begin{lemma}
  \label{lem:lift_antisymmetric}
  If~\(\Rmat\) is antisymmetric, then
  \((\UnivRmat)^*=\sigma(\UnivRmat)\) for the universal lift
  \(\UnivRmat\in\U(A^\univ\otimes A^\univ)\) constructed in
  Proposition~\textup{\ref{prop:univ_R_matrix}}.
\end{lemma}

\begin{proof}
  Both \(\sigma(\UnivRmat)^*\) and~\(\UnivRmat\) are bicharacters that
  lift~\(\Rmat\).  They must be equal because bicharacters lift
  uniquely by~\cite{Meyer-Roy-Woronowicz:Homomorphisms}*{Proposition
    4.7}.
\end{proof}

\begin{proposition}
  \label{pro:symmetric_braiding}
  The braiding on \(\Corepcat(\G)\) constructed from
  \(\Rmat\in\U(\hat{A}\otimes\hat{A})\) is symmetric if and only
  if~\(\Rmat\) is antisymmetric.
\end{proposition}

\begin{proof}
  Let \(\Hils_1\) and~\(\Hils_2\) be Hilbert spaces with
  corepresentations of~\(\G\).  Let \(\hat{\phi}_i\colon
  \hat{A}^\univ\to\Bound(\Hils_i)\) be the corresponding
  \Star{}representations.  Then
  \[
  \Hils_1\otimes\Hils_2
  \xrightarrow{\Braiding{\Hils_1}{\Hils_2}}
  \Hils_2\otimes\Hils_1
  \xrightarrow{\Braiding{\Hils_2}{\Hils_1}}
  \Hils_1\otimes\Hils_2
  \]
  is equal to
  \[
  (\hat{\phi}_1\otimes\hat{\phi}_2)(\UnivRmat)^*
  \circ \Flip^{(\Hils_2,\Hils_1)} \circ
  (\hat{\phi}_2\otimes\hat{\phi}_1)(\UnivRmat)^* \circ
  \Flip^{(\Hils_1,\Hils_2)} = (\hat{\phi}_1\otimes\hat{\phi}_2)
  (\flip(\UnivRmat)\UnivRmat)^*.
  \]
  This is the identity operator for all
  representations~\(\hat{\phi}_i\) if and only if
  \(\sigma(\UnivRmat)\UnivRmat = 1\).
\end{proof}

\subsection{The Abelian case}
\label{sec:braiding_Abelian}

Let~\(B\) be a locally compact group.  What is an \Rmattxt\nb-matrix
for the commutative quantum group \((\Contvin(G),\Comult)\)?  Since
\(\Contvin(G)\otimes\Contvin(G)\) is commutative as well,
\eqref{eq:reduced_R_mat_equivariant_cond} simplifies to the condition
\(\sigma\circ\Comult=\Comult\), which is equivalent to~\(G\) being
commutative.  Hence there is no \Rmattxt\nb-matrix unless~\(G\) is
Abelian, which we assume from now on.
Then~\eqref{eq:reduced_R_mat_equivariant_cond} holds for any unitary
\(\Rmat\in\U(\Contvin(G)\otimes\Contvin(G))\), so an
\Rmattxt\nb-matrix for~\(\G\) is simply a bicharacter.  Equivalently,
\(\Rmat\) is a function \(\rho\colon G\times G\to \Un(1)\) satisfying
\(\rho(xy,z)= \rho(x,z)\rho(y,z)\) and \(\rho(x,yz)=
\rho(x,y)\rho(x,z)\).  Being antisymmetric means
\(\rho(x,y)\rho(y,x)=1\) for all \(x,y\in G\).

Any bicharacter~\(\rho\) as above is of the form \(\rho(x,y) =
\langle\hat{\rho}(x),y\rangle\) for a group homomorphism
\(\hat{\rho}\colon G\to \hat{G}\) to the Pontrjagin dual~\(\hat{G}\),
with \(\rho(x,\blank)=\hat{\rho}\).  This is a special case of the
interpretation of bicharacters as quantum group homomorphisms
in~\cite{Meyer-Roy-Woronowicz:Homomorphisms}.

The category of Hilbert space representations of~\(G\) is equivalent
to the category of corepresentations of \((\Contvin(G),\Delta)\) and
to the category of representations of \(\Cst(G) \cong
\Contvin(\hat{G})\).  The tensor category of
\(G\)\nb-representations is already symmetric for the obvious
braiding~\(\Sigma\), which corresponds to the
\Rmattxt\nb-matrix~\(1\).  What are the braiding operators for a
nontrivial \Rmattxt\nb-matrix?

Let \(\int^\oplus_{\hat{G}} \Hils_x \,\diff\mu(x)\) denote the Hilbert
space of \(L^2\)\nb-sections of a measurable field of Hilbert
spaces~\((\Hils_x)_{x\in\hat{G}}\) over~\(\hat{G}\) with respect to a
measure~\(\mu\), equipped with the action of \(\Contvin(\hat{G})\) by
pointwise multiplication.  Any representation of~\(\Contvin(\hat{G})\)
is of this form, where~\(\mu\) is unique up to measure equivalence and
the field~\((\Hils_x)\) is unique up to isomorphism \(\mu\)\nb-almost
everywhere.  Let \(\Hils_1 = \int^\oplus_{\hat{G}} \Hils_{1_x}
\,\diff\mu_1(x)\) and \(\Hils_2 = \int^\oplus_{\hat{G}} \Hils_{2_x}
\,\diff\mu_2(x)\) be two Hilbert space representations of~\(G\).  Then
\[
\Hils_1 \otimes \Hils_2 =
\int^\oplus_{\hat{G}\times\hat{G}} \Hils_{1_x}\otimes\Hils_{2_y}
\,\diff\mu_1(x)\,\diff\mu_2(y)
\]
with \(\Contvin(\hat{G}) \otimes \Contvin(\hat{G}) \cong
\Contvin(\hat{G}\times\hat{G})\) acting by pointwise multiplication.
The braiding \(\Braiding{\Hils_1}{\Hils_2}\) maps an \(L^2\)\nb-section
\((\xi_{x,y})_{x,y}\) of the field
\((\Hils_{1_x}\otimes\Hils_{2_y})_{x,y}\) to the section
\((y,x)\mapsto \rho(y,x)^{-1}\xi_{x,y}\) of
\((\Hils_{2_y}\otimes\Hils_{1_x})_{y,x}\).

\begin{example}
  \label{exa:Ztwo}
  Consider \(G=\Z/2=\{\pm1\}\) and let \(\rho(x,y) = xy \in \Z/2
  \subseteq \Un(1)\); this bicharacter corresponds to the isomorphism
  \(G\cong\hat{G}\).  It is both symmetric and antisymmetric.  The
  spectral analysis above writes a \(\Z/2\)-Hilbert space as a
  \(\Z/2\)-graded Hilbert space, splitting it into even and odd
  elements with respect to the action of the generator of~\(\Z/2\).
  The braiding unitary on \(\xi\otimes\eta\) is~\(\Sigma\) if \(\xi\)
  or~\(\eta\) is even, and~\(-\Sigma\) if both \(\xi\) and~\(\eta\)
  are odd.  This is the usual Koszul sign rule.
\end{example}

\section{Coaction categories of quasitriangular quantum groups}
\label{sec:Braided_tens_Cst}

Let \(\G=(A,\Comult[A],\Rmat)\) be a quasitriangular quantum group.
Let \((C,\gamma)\) and \((D,\delta)\) be \(\G\)\nb-\(\Cst\)-algebras.
The twisted tensor product \(C\boxtimes_{\Rmat}D = C\boxtimes D\) is
constructed in~\cite{Meyer-Roy-Woronowicz:Twisted_tensor}.  It is a
crossed product of \(C\) and~\(D\), that is, there are canonical
morphisms \(\iota_C\colon C\to C \boxtimes D\) and \(\iota_D\colon
D\to C \boxtimes D\) with
\[
\iota_C(C)\cdot \iota_D(D) =
\iota_D(D)\cdot \iota_C(C) =
C \boxtimes D;
\]
here a \emph{morphism} is a nondegenerate \Star{}homomorphism to the
multiplier algebra and \(X\cdot Y\) for two subspaces \(X\) and~\(Y\)
of a \(\Cst\)\nb-algebra means the \emph{closed linear span} of
\(x\cdot y\) for \(x\in X\), \(y\in Y\) as
in~\cite{Meyer-Roy-Woronowicz:Twisted_tensor}.

Theorem~\ref{the:crossed_prod} recalls one of the two equivalent
definitions of the twisted tensor product
in~\cite{Meyer-Roy-Woronowicz:Twisted_tensor}.  Let
\((\varphi,\corep{U}^{\Hils})\) and \((\psi,\corep{U}^{\Hils[K]})\) be
faithful covariant representations of \((C,\gamma)\) and
\((D,\delta)\) on Hilbert spaces \(\Hils\) and~\(\Hils[K]\),
respectively.  Then \(C \boxtimes_{\Rmat}D\) is canonically isomorphic
to \(\phi_1(C)\cdot \tilde\psi_2(D)\subseteq
\Bound(\Hils\otimes\Hils[K])\), where \(\phi_1(c)=\phi(c)\otimes
1_{\Hils[K]}\) and \(\tilde\psi_2(d) = X(1_{\Hils}\otimes\psi(d))X^*\)
for the unitary~\(X\) that is characterised
by~\eqref{eq:commutation_between_corep_of_two_quantum_groups}.  The
same unitary appears in our construction of the braiding, so
\begin{equation}
  \label{eq:tilde_psi_via_braiding}
  \tilde\psi_2(d)
  = \Braiding{\Hils}{\Hils[K]} (\psi(d)\otimes1_{\Hils})
  (\Braiding{\Hils}{\Hils[K]})^*.
\end{equation}

We are going to equip \(C \boxtimes D\) with a natural
\(\G\)\nb-coaction and show that this tensor product gives a monoidal
structure on the category of
\(\G\)\nb-\(\Cst\)-algebras~\(\Cstcat(\G)\) (see
Definition~\ref{def:cont_coaction}).

\begin{proposition}
  \label{pro:braide_tens_presv_cov_corep}
  There is a unique \(\G\)\nb-coaction \(\gamma\bowtie_\Rmat\delta\)
  on~\(C\boxtimes_\Rmat D\) such that the canonical representation
  on~\(\Hils\otimes\Hils[K]\) and the corepresentation
  \(\corep{U}^{\Hils}\tenscorep \corep{U}^{\Hils[K]}\) form a
  covariant representation of \((C\boxtimes_\Rmat
  D,\gamma\bowtie_\Rmat\delta)\).  This coaction is also the unique
  one for which the morphisms \(\iota_C\colon C\to C\boxtimes_\Rmat
  D\) and \(\iota_D\colon D\to C\boxtimes_\Rmat D\) are
  \(\G\)\nb-equivariant.
\end{proposition}

\begin{proof}
  We identify \(C\boxtimes D\) with its image in
  \(\Bound(\Hils\otimes\Hils[K])\).  The covariance of this
  representation of~\(C\boxtimes D\) with
  \(\corep{U}^{\Hils}\tenscorep \corep{U}^{\Hils[K]}\) means that
  \[
  (\gamma\bowtie_\Rmat\delta)(x) =
  (\corep{U}^{\Hils}\tenscorep \corep{U}^{\Hils[K]}) (x\otimes 1_A)
  (\corep{U}^{\Hils}\tenscorep \corep{U}^{\Hils[K]})^*
  \]
  for all \(x\in C\boxtimes D\).  Hence there is at most one such
  coaction~\(\gamma\bowtie_\Rmat\delta\).

  The representation \(c\mapsto \iota_C(c) = \phi(c)\otimes
  1_{\Hils[K]}\) is covariant with respect to
  \(\corep{U}^{\Hils}\tenscorep \corep{U}^{\Hils[K]}\) because it is
  covariant with respect to \(\corep{U}^{\Hils}\otimes1\) by
  construction and~\(\iota_C(c)\) acts only on the first leg.  Hence
  \((\gamma\bowtie_\Rmat\delta)(\iota_C(c)) = (\iota_C\otimes
  \Id_A)\gamma(c)\) for all \(c\in C\).  Similarly, the representation
  \(d\mapsto \psi(d)\otimes1\) on \(\Hils[K]\otimes\Hils\) is
  covariant with respect to \(\corep{U}^{\Hils[K]}\tenscorep
  \corep{U}^{\Hils}\).  Since the
  unitary~\(\Braiding{\Hils}{\Hils[K]}\) is \(\G\)\nb-equivariant by
  Proposition~\ref{prop:braiding_quasi_qgrp_and_R_matrix}, the
  representation~\(\tilde\psi_2\) on \(\Hils\otimes\Hils[K]\) is
  covariant with respect to \(\corep{U}^{\Hils}\tenscorep
  \corep{U}^{\Hils[K]}\) as well (unlike the representation \(d\mapsto
  1\otimes\psi(d)\)).  Hence \((\gamma\bowtie_\Rmat\delta)(\iota_D(d))
  = (\iota_D\otimes \Id_A)\delta(d)\) for all \(d\in D\).  As a
  result, \(\gamma\bowtie_\Rmat\delta\) maps \(C\boxtimes D =
  \iota_C(C)\cdot \iota_D(D)\) nondegenerately into the multiplier
  algebra of \((C\boxtimes D)\otimes A\), and the morphisms
  \(\iota_C\) and~\(\iota_D\) are \(\G\)\nb-equivariant.

  The morphism \(\gamma\bowtie_\Rmat\delta\colon C\boxtimes D\to
  (C\boxtimes D)\otimes A\) is faithful by construction.  The Podle\'s
  condition for \(C\boxtimes D\) follows from those for \(C\)
  and~\(D\):
  \begin{multline*}
    (\gamma\bowtie_\Rmat\delta)(C\boxtimes D)\cdot (1\otimes A)
    = (\iota_C\otimes\Id_A)(\gamma(C)) \cdot
    (\iota_D\otimes\Id_A)(\delta(D)) \cdot (1\otimes A)
    \\= (\iota_C\otimes\Id_A)(\gamma(C)) \cdot (\iota_D(D)\otimes A)
    = (\iota_C\otimes\Id_A)(\gamma(C)) \cdot (1\otimes A)\cdot
    (\iota_D(D)\otimes A)
    \\= (\iota_C(C)\otimes A) \cdot (\iota_D(D)\otimes A)
    = C\boxtimes D\otimes A.
  \end{multline*}
  Thus~\(\gamma\bowtie_\Rmat\delta\) is a continuous
  \(\G\)\nb-coaction on~\(C\boxtimes D\) for which \(\iota_C\)
  and~\(\iota_D\) are equivariant.  Conversely, if \(\iota_C\)
  and~\(\iota_D\) are \(\G\)\nb-equivariant, then
  \[
  (\gamma\bowtie_\Rmat\delta)(\iota_C(c)\cdot \iota_D(d))
  = (\iota_C\otimes\Id_A)(\gamma(c))\cdot (\iota_D\otimes\Id_A)(\delta(d))
  \]
  for \(c\in C\), \(d\in D\); this determines
  \(\gamma\bowtie_\Rmat\delta\) because \(\iota_C(C)\cdot \iota_D(D) =
  C\boxtimes D\).
\end{proof}

\begin{proposition}
  \label{prop:twisted_tens_R-mat}
  The coaction~\(\gamma\bowtie_\Rmat\) on~\(C\boxtimes_\Rmat D\) is
  natural with respect to equivariant morphisms, that is, it gives a
  bifunctor \(\boxtimes_\Rmat\colon \Cstcat(\G)\times\Cstcat(\G)\to
  \Cstcat(\G)\).  It is the only natural coaction for which~\(\C\)
  with the obvious isomorphisms \(C\boxtimes\C\cong C\) and
  \(\C\boxtimes D\cong D\) is a tensor unit.
\end{proposition}

\begin{proof}
  Two \(\G\)\nb-equivariant morphisms \(f\colon C_1\to C_2\) and
  \(g\colon D_1\to D_2\) induce a morphism \(f\boxtimes g\colon
  C_1\boxtimes D_1 \to C_2\boxtimes D_2\), which is determined
  uniquely by the conditions \((f\boxtimes g)\circ \iota_{C_1} =
  \iota_{C_2}\circ f\) and \((f\boxtimes g)\circ \iota_{D_1} =
  \iota_{D_2}\circ g\)
  (see~\cite{Meyer-Roy-Woronowicz:Twisted_tensor}*{Lemma 5.5}).  The
  coactions \(\gamma_1\bowtie_\Rmat \delta_1\) and
  \(\gamma_2\bowtie_\Rmat \delta_2\) satisfy \((\gamma_k\bowtie_\Rmat
  \delta_k)\circ \iota_{C_k} = (\iota_{C_k}\otimes \Id_A)\circ
  \gamma_k\) and \((\gamma_k\bowtie_\Rmat \delta_k)\circ \iota_{D_k} =
  (\iota_{D_k}\otimes \Id_A)\circ \delta_k\) for \(k=1,2\).  Thus
  \[
  (\gamma_2\bowtie_\Rmat \delta_2)\circ \iota_{C_2} \circ f
  = (f\boxtimes g\otimes\Id_A)\circ (\gamma_1\bowtie_\Rmat
  \delta_1)\circ \iota_{C_1}
  \]
  and similarly on~\(D_1\).  So \(f\boxtimes g\) is equivariant
  and~\(\boxtimes_\Rmat\) is a bifunctor as asserted.  The obvious
  isomorphisms \(C\boxtimes\C\cong C\) and \(\C\boxtimes D\cong D\)
  are \(\G\)\nb-equivariant and natural and satisfy the triangle axiom
  for a tensor unit in a monoidal category; so~\(\C\) with these
  isomorphisms is a unit for the tensor product~\(\boxtimes_\Rmat\)
  on~\(\Cstcat(\G)\).

  Conversely, assume that \(\gamma \boxdot \delta\) is a natural
  \(\G\)\nb-coaction on \(C\boxtimes D\) for which~\(\C\) with the
  obvious isomorphisms \(C\boxtimes\C\cong C\) and \(\C\boxtimes
  D\cong D\) is a tensor unit.  That is, these two isomorphisms are
  \(\G\)\nb-equivariant.  The unique morphisms \(1_C\colon \C\to C\)
  and \(1_D\colon \C\to D\) given by the unit multiplier are
  equivariant with respect to the trivial \(\G\)\nb-coaction
  on~\(\C\).  We have
  \(\iota_C = \Id_C\boxtimes 1_D\) and \(\iota_D= 1_C\boxtimes
  \Id_D\).  Hence \(\iota_C\) and~\(\iota_D\) are equivariant for
  \(\gamma\boxdot\delta\).  This forces \(\gamma\boxdot\delta =
  \gamma\bowtie_\Rmat \delta\).
\end{proof}

\begin{theorem}
  \label{the:Braided_coaction_Cat}
  If \(C_1,C_2,C_3\) are objects of \(\Cstcat(\G)\), then there is a
  unique isomorphism of triple crossed products \(C_1 \boxtimes
  (C_2\boxtimes C_3)\cong (C_1 \boxtimes C_2)\boxtimes C_3\), which
  is also \(\G\)\nb-equivariant.  Thus \(\Cstcat(\G)\) with the
  tensor product~\(\boxtimes_\Rmat\) is a monoidal category.
\end{theorem}

\begin{proof}
  An isomorphism of triple crossed products is an isomorphism that
  intertwines the embeddings of \(C_1\), \(C_2\) and~\(C_3\).  Since
  the images of these embeddings generate the crossed product, such an
  isomorphism is unique if it exists.

  Let \((C_i,\gamma_i)\) be \(\G\)\nb-\(\Cst\)-algebras and let
  \((\varphi_i,\corep{U}^{\Hils_i})\) be faithful covariant
  representations
  of \((C_i,\gamma_i)\), respectively, for \(i=1,2,3\).  The
  construction of the \(\G\)\nb-coaction on \(C_i\boxtimes C_j\)
  shows that \((\varphi_i\boxtimes\varphi_j,
  \corep{U}^{\Hils_i}\tenscorep \corep{U}^{\Hils_j})\) is a faithful
  covariant representation of~\(C_i\boxtimes C_j\)
  on~\(\Hils_i\otimes\Hils_j\).  Therefore,
  Theorem~\ref{the:crossed_prod} gives a faithful representation
  \(\varphi_1\boxtimes (\varphi_2\boxtimes\varphi_3)\)
  of~\(C_1 \boxtimes(C_2\boxtimes C_3)\) on
  \(\Hils_1\otimes\Hils_2\otimes\Hils_3\), which is characterised by:
  \begin{alignat*}{2}
    \iota_{C_1}(c_1)
    &\mapsto \varphi_1(c_1)\otimes 1_{\Hils_2}\otimes 1_{\Hils_3},\\
    \iota_{C_2}(c_2)
    &\mapsto (\Braiding{\Hils_2\otimes\Hils_3}{\Hils_1})
    \bigl(\varphi_2(c_2)\otimes 1_{\Hils_3}\otimes 1_{\Hils_1}\bigr)
    (\Braiding{\Hils_2\otimes\Hils_3}{\Hils_1})^*,\\
    \iota_{C_3}(c_3)
    &\mapsto (\Braiding{\Hils_2\otimes\Hils_3}{\Hils_1})
    (\Braiding{\Hils_3}{\Hils_2}_{12})
    \bigl(\varphi_3(c_3)\otimes 1_{\Hils_2}\otimes 1_{\Hils_1}\bigr)
    (\Braiding{\Hils_3}{\Hils_2}_{12})^*
    (\Braiding{\Hils_2\otimes\Hils_3}{\Hils_1})^*
  \end{alignat*}
  for \(c_i\in C_i\), \(i=1,2,3\).  The diagrams in
  Proposition~\ref{prop:braiding_quasi_qgrp_and_R_matrix} and
  Corollary~\ref{cor:Yang-Baxter_rmat_braid} give
  \begin{gather*}
    \Braiding{\Hils_2\otimes\Hils_3}{\Hils_1}
    = \Braiding{\Hils_2}{\Hils_1}_{12} \cdot
    \Braiding{\Hils_3}{\Hils_1}_{23},\\
    \Braiding{\Hils_2\otimes\Hils_3}{\Hils_1}\cdot
    \Braiding{\Hils_3}{\Hils_2}_{12}
    = \Braiding{\Hils_2}{\Hils_1}_{12} \cdot
    \Braiding{\Hils_3}{\Hils_1}_{23} \cdot
    \Braiding{\Hils_3}{\Hils_2}_{12}
    = \Braiding{\Hils_3}{\Hils_1\otimes\Hils_2} \cdot
    \Braiding{\Hils_2}{\Hils_1}_{23}.
  \end{gather*}
  Hence the above characterisation of \(\varphi_1\boxtimes
  (\varphi_2\boxtimes\varphi_3)\) simplifies to
  \begin{equation}
    \label{eq:present_C1_tens_C2_C3_aux}
    \begin{alignedat}{2}
      \bigl(\varphi_1\boxtimes (\varphi_2\boxtimes\varphi_3)\bigr)
      \circ\iota_{C_1}(c_1)
      &= \varphi_1(c_1)\otimes 1,\\
      \bigl(\varphi_1\boxtimes(\varphi_2\boxtimes\varphi_3)\bigr)
      \circ\iota_{C_2}(c_2)
      &= \Braiding{\Hils_2}{\Hils_1}_{12}(\varphi_2(c_2)\otimes 1)
      (\Braiding{\Hils_2}{\Hils_1}_{12})^*,\\
      \bigl(\varphi_1\boxtimes(\varphi_2\boxtimes\varphi_3)\bigr)
      \circ\iota_{C_3}(c_3)
      &= (\Braiding{\Hils_3}{\Hils_1\otimes\Hils_2})
      (\varphi_3(c_3)\otimes 1)
      (\Braiding{\Hils_3}{\Hils_1\otimes\Hils_2})^*.
    \end{alignedat}
  \end{equation}

  Similarly, we get a faithful representation
  \((\varphi_1\boxtimes\varphi_2)\boxtimes\varphi_3\) of
  \((C_1\boxtimes C_2)\boxtimes C_3\) on~\(\Hils_1\otimes
  \Hils_2\otimes\Hils_3\).  A computation as above shows that the
  combinations of braiding unitaries in it are equal to those
  in~\eqref{eq:present_C1_tens_C2_C3_aux}.  Thus
  \(\varphi_1\boxtimes(\varphi_2\boxtimes\varphi_3)\) and
  \((\varphi_1\boxtimes\varphi_2)\boxtimes\varphi_3\) are the same
  representation on~\(\Hils_1\otimes\Hils_2\otimes\Hils_3\).  This
  gives an isomorphism \(C_1\boxtimes (C_2\boxtimes C_3) \cong
  (C_1\boxtimes C_2)\boxtimes C_3\) that intertwines the canonical
  embeddings of \(C_1\), \(C_2\) and~\(C_3\).  It is
  \(\G\)\nb-equivariant because our \(\G\)\nb-coactions are uniquely
  determined by their actions on the tensor factors \(C_1\), \(C_2\)
  and~\(C_3\).

  The natural isomorphisms above provide the associators needed for
  a monoidal category.  The pentagon condition for these associators
  and the compatibility with the unit transformations \(\C\boxtimes
  D\cong D\), \(C\boxtimes \C\cong C\) follow by checking the
  relevant commuting diagram on each tensor factor separately.
\end{proof}

% \begin{proposition}
%  \label{prop:property_monoidal}
%   The monoidal functor \(\boxtimes_\Rmat\) has the
%   following properties:
%   \begin{enumerate}
%   \item \(C\boxtimes_\Rmat D\) is a crossed product of \(C\)
%     and~\(D\), that is, \(C\boxtimes_\Rmat
%     D=\iota_C(C)\cdot\iota_D(D)\);
%   \item the \(\boxtimes_\Rmat\)\nb-products of injective
%     \(\G\)\nb-equivariant morphisms are injective;
%   \item if one of the actions \(\gamma\) and~\(\delta\) is trivial,
%     then \(C\boxtimes_\Rmat D \cong C\otimes D\) as crossed products;
%     that is, \(\iota_C(c)\) and \(\iota_D(d)\) are mapped to
%     \(c\otimes 1_D\) and \(1_C\otimes d\) for \(c\in C\), \(d\in D\).
%  \end{enumerate}
% \end{proposition}

% \begin{proof}
%   Theorem~\ref{the:crossed_prod} gives the first property.  The second
%   and third properties follow from Proposition~5.6 and Lemma~5.11
%   in~\cite{Meyer-Roy-Woronowicz:Twisted_tensor}, respectively.
% \end{proof}

\begin{proposition}
  \label{pro:braided_monoidal_Cstar}
  The monoidal structure~\(\boxtimes_\Rmat\) is braided monoidal if
  and only if it is symmetric monoidal, if and only if~\(\Rmat\) is
  antisymmetric.  In that case, the braiding is the unique
  isomorphism of crossed products \((C\boxtimes_\Rmat
  D,\iota_C,\iota_D) \cong (D\boxtimes_\Rmat C,\iota_C,\iota_D)\).
\end{proposition}

\begin{proof}
  \cite{Meyer-Roy-Woronowicz:Twisted_tensor}*{Proposition 5.1} shows
  that \(D\boxtimes_\Rmat C \cong C\boxtimes_{\hat{\Rmat}} D\) as
  crossed products, where \(\hat{\Rmat} \defeq \sigma(\Rmat^*)\).
  If~\(\Rmat\) is antisymmetric, this gives an isomorphism of crossed
  products between \((C\boxtimes_\Rmat D,\iota_C,\iota_D)\) and
  \((D\boxtimes_\Rmat C,\iota_C,\iota_D)\).  This isomorphism is
  equivariant because our coactions are determined by what they do on
  the embedded copies of \(C\) and~\(D\).  Thus we get a braided
  monoidal category in this case.

  Conversely, any braiding must give the identity map on
  \(C\boxtimes_\Rmat \C\) and \(\C\boxtimes_\Rmat D\) because~\(\C\)
  is the tensor unit.  Since \(\iota_C = \Id_C\boxtimes 1_D\) and
  \(\iota_D = 1_C\boxtimes \Id_D\), any braiding must be an
  isomorphism of crossed products \(C\boxtimes_\Rmat D \cong
  D\boxtimes_\Rmat C\).  By the argument above, this happens if and
  only if \(C\boxtimes_{\Rmat} D = C\boxtimes_{\hat{\Rmat}} D\) as
  crossed products.  Corollary~\ref{cor:tensor_determines_bichar} says
  that the crossed product~\(A\boxtimes_{\Rmat}A\)
  determines~\(\Rmat\) uniquely.  Hence \(\Rmat=\hat{\Rmat}\).
\end{proof}

\section{Quantum codoubles}
\label{sec:coact_cat_codoub}

Quantum codoubles of compact quantum groups were introduced by
Podle\'s and Woronowicz
in~\cite{Podles-Woronowicz:Quantum_deform_Lorentz} to construct an
example of a quantum Lorentz group.  The definition was extended to
the non-compact case
in~\cite{Woronowicz-Zakrzewski:Quantum_Lorentz_Gauss}.  Quantum
codoubles were also described by Baaj and
Vaes in~\cite{Baaj-Vaes:Double_cros_prod}*{Proposition 9.5},
assuming the underlying quantum group to be generated by a regular
multiplicative unitary.  We call the dual of the quantum codouble
Drinfeld double.  Some authors use different notation, exchanging
doubles and codoubles.

We shall refer to~\cite{Roy:Codoubles} for the definition of the
quantum codouble~\(\Codouble{\G}\) and the Drinfeld
double~\(\Drinfdouble{\G}\) of a \(\Cst\)\nb-quantum group
\(\G=(A,\Comult[A])\).  It is shown in~\cite{Roy:Codoubles} that
\(\Codouble{\G}\) and \(\Drinfdouble{\G}\) are again
\(\Cst\)\nb-quantum groups and that \(\Codouble{\G}\)-coactions on
\(\Cst\)\nb-algebras are equivalent to \(\G\)\nb-Yetter--Drinfeld
\(\Cst\)\nb-algebras.  To simplify notation, we shall only consider
the special case of the results in~\cite{Roy:Codoubles} where
\(B=\hat{A}\) and the bicharacter~\(\textup{V}\) is
\(\multunit\in\U(\hat{A}\otimes A)\) because that is all we need
below.

The main result in this section is that this equivalence of
categories between \(\Codouble{\G}\)-coactions and
\(\G\)\nb-Yetter--Drinfeld \(\Cst\)\nb-algebras turns the tensor
products \(\boxtimes_{\Rmat}\) for
\(\Codouble{\G}\)-\(\Cst\)\nb-algebras and a canonical
\(\Rmat\)\nb-matrix for \(\Codouble{\G}\) into the tensor
product~\(\boxtimes_{\multunit}\) for \(\G\)\nb-Yetter--Drinfeld
algebras.  We also show that the tensor product~\(\boxtimes_\Rmat\)
for a general quasitriangular quantum group is a special case of the
same operation for its codouble (see
Theorem~\ref{the:Codoub_quasi_triag}).

The \emph{quantum codouble} \(\Codouble{\G}=\Bialg{\CodoubAlg}\)
of~\(\G\) is defined by \(\CodoubAlg\defeq A\otimes\hat{A}\) and
\begin{align*}
  \flip^{\multunit}\colon A\otimes\hat{A} &\to\hat{A}\otimes A,&\qquad
  a\otimes\hat{a} &\mapsto \multunit(\hat{a}\otimes a)\multunit[*],\\
  \Comult[\CodoubAlg]\colon\CodoubAlg&\to\CodoubAlg\otimes\CodoubAlg,
  &\qquad
  a\otimes\hat{a}&\mapsto
  \flip^{\multunit}_{23}(\Comult[A](a)\otimes\DuComult[A](\hat{a})),
\end{align*}
for \(a\in A\), \(\hat{a}\in\hat{A}\).  We may
generate~\(\Codouble{\G}\) by a manageable multiplicative unitary by
\cite{Roy:Codoubles}*{Theorem 4.1}.  So it is a \(\Cst\)\nb-quantum
group and has a dual \(\Drinfdouble{\G}=\Bialg{\DrinfDoubAlg}\), which
is called the \emph{Drinfeld double} of~\(\G\).  We have
\[
\DrinfDoubAlg = \rho(A)\cdot\theta(\hat{A})\quad\text{and}\quad
\Comult[\DrinfDoubAlg](\rho(a)\cdot\theta(\hat{a})) =
(\rho\otimes\rho)\Comult[A](a)\cdot(\theta\otimes\theta)
\DuComult[A](\hat{a})
\]
for a certain pair of representations \(\rho\) and~\(\theta\) of~\(A\)
and~\(\hat{A}\) on the same Hilbert space.  The formulas for \(\rho\)
and~\(\theta\) will not be needed in the following.

It is crucial that \(\rho\) and~\(\theta\) give Hopf
\Star{}homomorphisms from \(\G\) and~\(\DuG\) to \(\Drinfdouble{\G}\).
These induce dual morphisms \(\Codouble{\G} \to \G\) and
\(\Codouble{\G} \to \DuG\) (compare
Theorem~\ref{the:equivalent_notion_of_homomorphisms}).  These quantum
group morphisms induce a map on corepresentations (see
\cite{Meyer-Roy-Woronowicz:Homomorphisms}*{Proposition 6.5} and the
proof of \cite{Meyer-Roy-Woronowicz:Twisted_tensor}*{Theorem 5.2} for
a correct general proof).  Thus a corepresentation
of~\(\Codouble{\G}\) induces corepresentations of \(\G\) and~\(\DuG\)
on the same Hilbert space.  It is shown in~\cite{Roy:Codoubles} that
this gives a bijection between corepresentations of~\(\Codouble{\G}\)
and certain pairs of corepresentations of \(\G\) and~\(\DuG\):

\begin{proposition}[\cite{Roy:Codoubles}*{Proposition 6.11}]
  \label{prop:corep_codoub_classify}
  Let~\(\Hils[K]\) be a Hilbert space.  The corepresentations
  \(\corep{U}\in\U(\Comp(\Hils[K])\otimes A)\) and
  \(\corep{V}\in\U(\Comp(\Hils[K])\otimes\hat{A})\) of \(\G\)
  and~\(\DuG\) associated to a corepresentation \(\corep{X}\in
  \U(\Comp(\Hils[K])\otimes\widehat{\DrinfDoubAlg})\)
  of~\(\Codouble{\G}\) satisfy
  \[
  \flip^{\multunit}_{23}\bigl(\corep{U}_{12}\corep{V}_{13}\bigr)
  = \corep{V}_{12}\corep{U}_{13}
  \qquad\text{in }\U(\Comp(\Hils[K])\otimes\hat{A}\otimes A);
  \]
  we call a pair~\((\corep{U},\corep{V})\) with this property
  \emph{\(\Codouble{\G}\)\nb-compatible}.  The map \(\corep{X}\mapsto
  (\corep{U},\corep{V})\) above is a bijection from corepresentations
  of~\(\Codouble{\G}\) to \(\Codouble{\G}\)\nb-compatible pairs of
  corepresentations of \(\G\) and~\(\DuG\), with inverse
  \[
  \corep{X}\defeq\corep{U}_{12}\corep{V}_{13}
  \qquad\text{in }\U(\Comp(\Hils[K])\otimes A\otimes\hat{A}).
  \]
\end{proposition}

A quantum group morphism also induces a functor between the coaction
categories, see Theorem~\ref{the:equivalent_notion_of_homomorphisms}.
Thus a continuous coaction of~\(\Codouble{\G}\) on a
\(\Cst\)\nb-algebra~\(C\) induces coactions of \(\G\) and~\(\DuG\).
Once again, this gives a bijection from coactions of~\(\Codouble{\G}\)
to certain pairs of coactions of \(\G\) and~\(\DuG\):

\begin{proposition}
  \label{prop:G_Yetter_Drinfeld}
  Let~\(C\) be a \(\Cst\)\nb-algebra.  The continuous coactions
  \(\gamma\) and~\(\delta\) of \(\G\) and~\(\DuG\) associated to a
  continuous coaction~\(\xi\) of~\(\Codouble{\G}\) satisfy
  \begin{equation}
    \label{eq:G-Yetter_Drinfeld}
    \flip^{\multunit}_{23}\bigl((\gamma\otimes \Id_{\hat{A}})
    \delta\bigr)
    = (\delta\otimes\Id_A)\gamma;
  \end{equation}
  A \(\Cst\)\nb-algebra with such a pair of coactions is called a
  \(\G\)\nb-\emph{Yetter--Drinfeld \(\Cst\)-algebra}.

  The map \(\xi\mapsto (\gamma,\delta)\) above is a bijection from
  continuous coactions of~\(\Codouble{\G}\) to the set of pairs of
  continuous coactions \((\gamma,\delta)\)
  satisfying~\eqref{eq:G-Yetter_Drinfeld}; the inverse maps
  \((\gamma,\delta)\) to \(\xi= (\gamma\otimes\Id_{\hat{A}}) \delta\).
\end{proposition}

Yetter-Drinfeld \(\Cst\)-algebras were defined by Nest and Voigt in
\cite{Nest-Voigt:Poincare}*{Definition 3.1} (assuming Haar weights
on~\(\G\)), and Proposition~\ref{prop:G_Yetter_Drinfeld} is
essentially \cite{Nest-Voigt:Poincare}*{Proposition 3.2}.  For
\(\Cst\)\nb-quantum groups without Haar weights,
Proposition~\ref{prop:G_Yetter_Drinfeld} is
\cite{Roy:Codoubles}*{Proposition 6.8}, with an explicit description
of the bijection taken from the proof of
\cite{Roy:Codoubles}*{Proposition 6.8}.

Let~\(\YDcat(\G)\) denote the category with \(\G\)\nb-Yetter--Drinfeld
\(\Cst\)\nb-algebras as objects and morphisms that are both \(\G\)-
and \(\DuG\)\nb-equivariant as arrows.

The following unitary is an \Rmattxt\nb-matrix for~\(\Codouble{\G}\)
by \cite{Roy:Codoubles}*{Lemma~5.11}:
\[
\Rmat=(\theta\otimes\rho)\multunit
\in\U(\DrinfDoubAlg\otimes\DrinfDoubAlg).
\]
Thus~\(\Codouble{\G}\) is quasitriangular and the construction in the
previous section gives a monoidal
structure~\(\boxtimes_{\Rmat}\) on \(\Cstcat(\Codouble{\G})\).  What
happens when we translate this to the equivalent setting of
\(\G\)\nb-Yetter--Drinfeld \(\Cst\)\nb-algebras?

The reduced bicharacter \(\multunit[A]\in\U(\hat{A}\otimes A)\), being
a bicharacter, gives a tensor product~\(\boxtimes_{\multunit}\) for
two \(\G\)\nb-Yetter--Drinfeld \(\Cst\)\nb-algebras.  This tensor
product is also used by Nest and Voigt in~\cite{Nest-Voigt:Poincare}
(they require, however, that~\(\G\) has Haar weights).

\begin{theorem}
  \label{the:Cat_equiv_YD_and_codoub}
  Let \(C_1\) and~\(C_2\) be \(\Codouble{\G}\)-\(\Cst\)-algebras, view
  them also as \(\G\)\nb-Yetter--Drinfeld \(\Cst\)-algebras.  There is
  an equivariant isomorphism of crossed products
  \(C_1\boxtimes_{\Rmat} C_2 \cong C_1\boxtimes_{\multunit} C_2\).
\end{theorem}

\begin{proof}
  First we describe the braiding on \(\Corepcat(\Codouble{\G})\)
  induced by~\(\Rmat\) in terms of~\(\multunit\) and
  compatible pairs of corepresentations of \(A\) and~\(\hat{A}\).

  Recall that the maps \(\rho\colon A\to\DrinfDoubAlg\) and
  \(\theta\colon \hat{A}\to\DrinfDoubAlg\) are Hopf
  \Star{}homomorphisms.  Thus they lift to the universal quantum
  groups: \(\rho^\univ\colon A^\univ\to\DrinfDoubAlg^\univ\) and
  \(\theta^\univ\colon \hat{A}^\univ\to\DrinfDoubAlg^\univ\).  Let
  \(\multunit{}^\univ\in\U(\hat{A}^\univ\otimes A^\univ)\) be the
  universal lift of~\(\multunit\).  The unitary
  \((\theta^\univ\otimes\rho^\univ)(\multunit{}^\univ) \in
  \U(\DrinfDoubAlg^\univ\otimes\DrinfDoubAlg^\univ)\) is a bicharacter
  and lifts~\(\Rmat\).  Hence it is \emph{the} universal
  lift~\(\UnivRmat\) of~\(\Rmat\).

  A corepresentation of~\(\Codouble{\G}\) is equivalent to a
  representation of~\(\DrinfDoubAlg^\univ\).  Composing this with the
  morphisms \(\theta^\univ\) and~\(\rho^\univ\) gives representations
  \(\hat\pi\) and~\(\pi\) of \(\hat{A}^\univ\) and~\(A^\univ\).  These
  are, in turn, equivalent to corepresentations \(\corep{U}\)
  and~\(\corep{V}\) of~\(A\) and \(\hat{A}\).  The construction of
  \((\corep{U},\corep{V})\) is exactly the bijection to
  \(\Codouble{\G}\)\nb-compatible pairs of corepresentations in
  Proposition~\ref{prop:corep_codoub_classify}.

  Now take two corepresentations of~\(\Codouble{\G}\) on Hilbert
  spaces~\(\Hils_k\).  These correspond to representations~\(\Pi_k\)
  of~\(\DrinfDoubAlg^\univ\), which determine representations
  \(\hat\pi_k=\Pi_k\circ\theta^\univ\) and
  \(\pi_k=\Pi_k\circ\rho^\univ\) of \(\hat{A}^\univ\)
  and~\(A^\univ\) on~\(\Hils_k\) for \(k=1,2\).  The braiding
  unitary~\(\Braiding{\Hils_1}{\Hils_2}\) is given by
  \eqref{eq:Yang_baxter_unit_corep}
  and~\eqref{eq:braiding_induced_by_R-matrix} and involves the
  unitary
  \[
  (\Pi_1\otimes\Pi_2)(\UnivRmat)^*
  = (\Pi_1\theta^\univ\otimes \Pi_2\rho^\univ)(\multunit{}^\univ)^*
  = (\hat\pi_1\otimes\pi_2)(\multunit{}^\univ)^*.
  \]

  Let \((C_i,\lambda_i)\) be \(\Codouble{\G}\)\nb-\(\Cst\)-algebras.
  Proposition~\ref{prop:G_Yetter_Drinfeld} gives a unique pair of
  coactions \(\gamma_i\colon C_i\to C_i\otimes A\) and
  \(\delta_i\colon C_i\to C_i\otimes\hat{A}\) such that
  \((C_i,\gamma_i,\delta_i)\) is a \(\G\)\nb-Yetter\nb-Drinfeld
  \(\Cst\)-algebra and
  \(\lambda_i=(\gamma_i\otimes\Id_{\hat{A}})\circ\delta_i\) for
  \(i=1,2\).  There are faithful covariant representations
  \((\corep{X}^{\Hils_i},\varphi_i)\) of
  \((C_i,\lambda_i,\widehat{\DrinfDoubAlg})\) on Hilbert
  spaces~\(\Hils_i\) for \(i=1,2\).  We use these faithful
  covariant representations to define \(C_1\boxtimes_{\Rmat}C_2\) as a
  \(\Cst\)\nb-subalgebra of \(\Bound(\Hils_1\otimes\Hils_2)\),
  see Theorem~\ref{the:crossed_prod}.

  Proposition~\ref{prop:corep_codoub_classify} turns
  \(\corep{X}^{\Hils_i}\) into a \(\Codouble{\G}\)\nb-compatible
  pair of corepresentations
  \((\corep{U}^{\Hils_i},\corep{V}^{\Hils_i})\).  The maps on
  corepresentations and coactions induced by a quantum group morphism
  preserve covariance of representations.  Hence
  \((\varphi_i,\corep{U}^{\Hils_i})\) is a covariant representation
  of~\((C_i,\gamma_i,A)\) on~\(\Hils_i\) and
  \((\varphi_i,\corep{V}^{\Hils_i})\) is a covariant representation
  of~\((C_i,\delta_i,\hat{A})\) on~\(\Hils_i\) for \(i=1,2\),
  respectively.  We use these faithful covariant representations to
  define \(C_1\boxtimes_{\multunit}C_2\) as a \(\Cst\)\nb-subalgebra
  of \(\Bound(\Hils_1\otimes\Hils_2)\), see
  Theorem~\ref{the:crossed_prod}.

  The representation of \(C_1\boxtimes_{\bichar} C_2\)
  on \(\Hils_1\otimes\Hils_2\) comes from the representations
  \[
  C_1\ni c_1\mapsto \varphi_1(c_1)\otimes 1,\qquad
  C_2\ni c_2\mapsto Z_\bichar (1\otimes \varphi_2(c_2)) Z_\bichar^*,
  \]
  where \(Z_{\Rmat}=(\Pi_1\otimes\Pi_2)(\UnivRmat)^*\) and
  \(Z_{\multunit}= (\hat\pi_1\otimes\pi_2)(\multunit{}^\univ)^*\).
  The computation above shows that \(Z_{\multunit}=Z_{\Rmat}\), so
  \(C_1\boxtimes_{\Rmat} C_2 = C_1\boxtimes_{\multunit} C_2\) as
  \(\Cst\)\nb-subalgebras of \(\Bound(\Hils_1\otimes\Hils_2)\).
\end{proof}

Since \(\G\)\nb-Yetter--Drinfeld \(\Cst\)-algebras are equivalent to
\(\Codouble{\G}\)-\(\Cst\)-algebras, the isomorphism in
Theorem~\ref{the:Cat_equiv_YD_and_codoub} shows that
\(C_1\boxtimes_{\multunit} C_2\) for two \(\G\)\nb-Yetter--Drinfeld
\(\Cst\)\nb-algebras \(C_1\) and~\(C_2\) carries a unique
\(\G\)\nb-Yetter--Drinfeld \(\Cst\)-algebra structure for which the
embeddings of \(C_1\) and~\(C_2\) are equivariant.  This extra
structure is natural, and \((\YDcat(\G),\boxtimes_{\multunit})\) is a
monoidal category.  By construction, the equivalence between
\(\YDcat(\G)\) and \(\Cstcat(\Codouble{\G})\) is an equivalence of
monoidal categories between \((\YDcat(\G),\boxtimes_{\multunit})\) and
\((\Cstcat(\Codouble{\G}),\boxtimes_{\Rmat})\).

\begin{remark}
  \label{pro:Symmetric_braid_and_self_duality}
  Propositions \ref{pro:symmetric_braiding}
  and~\ref{pro:braided_monoidal_Cstar} show that~\(\boxtimes_\Rmat\)
  admits a braiding if and only if it is symmetric, if and only if the
  braiding on \(\Corepcat(\Codouble{\G})\) associated to~\(\Rmat\) is
  symmetric if and only if~\(\Rmat\) is antisymmetric, that is,
  \(\Rmat^* = \sigma(\Rmat)\).  For the codouble, this is equivalent
  to \(\multunit=\Dumultunit\).  We know no non-trivial multiplicative
  unitary with this property.  Since \(\multunit=\Dumultunit\) implies
  that \(A\) and~\(\hat{A}\) are the same \(\Cst\)\nb-algebra, such a
  multiplicative unitary cannot be regular.
\end{remark}

\begin{example}
  \label{exa:Heisenberg_double}
  The tensor product \(A\boxtimes_{\multunit}\hat{A}\) is the
  \emph{canonical Heisenberg double} of a \(\Cst\)\nb-quantum
  group~\(\G\), in the sense that its representations are the
  Heisenberg pairs of~\(\G\) (see
  Proposition~\ref{pro:rep_cros_vs_heis}).
  Theorem~\ref{the:Cat_equiv_YD_and_codoub} says that
  \(A\boxtimes_{\multunit}\hat{A} \cong A\boxtimes_{\Rmat}\hat{A}\);
  this is a \(\Cst\)\nb-algebraic version of
  \cite{Delvaux-vanDaele:Drinf_vs_Heisenberg_doub}*{Proposition 5.1}.
\end{example}

It is already shown in \cite{Nest-Voigt:Poincare}*{Section 3}
that~\(\YDcat(\G)\) is a monodical category for the tensor
product~\(\boxtimes_{\multunit}\) if~\(\G\) has Haar weights.

Now let~\(\G\) be a quasitriangular quantum group with \Rmattxt\nb-matrix
\(\Rmat\in \U(\hat{A}\otimes\hat{A})\).  We view~\(\Rmat\) as a quantum
group morphism from~\(A\) to~\(\hat{A}\).
Theorem~\ref{the:equivalent_notion_of_homomorphisms} explains
how~\(\Rmat\) gives an induced coaction \(\delta\colon C\to
C\otimes\hat{A}\) on any \(\G\)\nb-\(\Cst\)\nb-algebra~\((C,\gamma)\).

\begin{lemma}
  \label{lemm:G-Yetter_Drinfeld_from_R_mat}
  Any pair~\((\gamma,\delta)\) as above is \(\G\)\nb-Yetter--Drinfeld.
\end{lemma}

\begin{proof}
  Any object \(C\in\Cstcat(\G)\) is equivariantly isomorphic to a
  subobject of~\(D\otimes A\) with coaction~\(\Id_D\otimes\Comult[A]\)
  for some \(\Cst\)\nb-algebra~\(D\) by
  \cite{Meyer-Roy-Woronowicz:Twisted_tensor}*{Lemma~2.9}.  Since the
  tensor factor~\(D\) causes no problems, it suffices to prove the
  lemma for \((C,\gamma)= (A,\Comult[A])\); here \(\delta = \Delta_R\)
  is the right coaction characterised
  by~\eqref{eq:def_V_via_right_homomorphism} for \(\bichar=\Rmat\).

  The relation~\eqref{eq:equiv_red_R_mat_equivariant_cond} has to be
  modified for \(\Rmat\in\U(\hat{A}\otimes\hat{A})\) by taking the
  dual multiplicative unitaries because we use~\(\hat{A}\) instead
  of~\(A\).  This gives
  \begin{equation}
    \label{eq:R_mat_equiv_cond}
    \Rmat_{12}\multunit_{13}\multunit_{23}=\multunit_{23}\multunit_{13}\Rmat_{12}
    \qquad\text{in }\U(\hat{A}\otimes\hat{A}\otimes A).
  \end{equation}
  Equations \eqref{eq:def_V_via_right_homomorphism} for~\(\Rmat\),
  \eqref{eq:W_char_in_second_leg} and~\eqref{eq:R_mat_equiv_cond} give
  \begin{multline*}
    \flip^{\multunit}_{34}\bigl((\Id_{\hat{A}}\otimes
    (\Comult[A]\otimes\Id_{\hat{A}})\Delta_R)\multunit\bigr)
    = \multunit_{12} \flip^{\multunit}_{34}(\multunit_{13}\Rmat_{14})
    \\ = \multunit_{12} \multunit_{34}\multunit_{14}\Rmat_{13}\multunit[*]_{34}
    = \multunit_{12}\Rmat_{13}\multunit_{14}
    = (\Id_{\hat{A}}\otimes(\Delta_R\otimes\Id_A)\Comult[A])\multunit.
  \end{multline*}
  Finally, we slice the first leg of this equation with
  \(\omega\in\hat{A}'\).  This gives~\eqref{eq:G-Yetter_Drinfeld} for
  the pair~\((\Comult[A],\Delta_R)\) because slices of~\(\multunit\)
  generate~\(A\).
\end{proof}

Since a morphism between two \(\G\)\nb-\(\Cst\)-algebras is
\(\G\)\nb-equivariant if and only if it is \(\G\)- and
\(\DuG\)\nb-equivariant,
Lemma~\ref{lemm:G-Yetter_Drinfeld_from_R_mat} gives a fully faithful
embedding of~\(\Cstcat(\G)\) into \(\YDcat(\G)\cong
\Cstcat(\Codouble{\G})\) that leaves the underlying
\(\Cst\)\nb-algebras unchanged.  Thus an \(\Rmat\)\nb-matrix
for~\(\G\) induces a quantum group morphism \(\G\to\Codouble{\G}\).

\begin{theorem}
  \label{the:Codoub_quasi_triag}
  The embedding
  \((\Cstcat(\G),\boxtimes_{\Rmat})\to\YDcat(\G),\boxtimes_{\multunit})\)
  is monoidal.
\end{theorem}

\begin{proof}
  Let \(C\) and~\(D\) be two \(\G\)\nb-\(\Cst\)-algebras, equip them
  with the \(\G\)\nb-Yetter--Drinfeld structure described above.  As
  a \(\Cst\)\nb-algebra, we have an isomorphism of crossed products
  \(C\boxtimes_{\Rmat} D \cong C\boxtimes_{\multunit} D\) by
  \cite{Meyer-Roy-Woronowicz:Twisted_tensor}*{Example 5.4}.  The
  induced \(\G\)\nb-Yetter--Drinfeld algebra structure on
  \(C\boxtimes_{\multunit} D\) is the unique one for which the
  embeddings of \(C\) and~\(D\) are equivariant: combine
  Proposition~\ref{pro:braide_tens_presv_cov_corep} with the
  equivalence of \(\G\)\nb-Yetter--Drinfeld algebra structures and
  \(\Codouble{\G}\)-coactions.  Similarly, the induced
  \(A\)\nb-coaction on \(C\boxtimes_{\Rmat} D\) is the unique one
  for which the embeddings of \(C\) and~\(D\) are
  \(A\)\nb-equivariant.  Since the \(\hat{A}\)\nb-coactions
  constructed from~\(\Rmat\) and an \(A\)\nb-coaction are natural,
  the embeddings of \(C\) and~\(D\) into \(C\boxtimes_{\Rmat} D\)
  are also \(\hat{A}\)\nb-equivariant.  Hence the
  \(\G\)\nb-Yetter--Drinfeld algebra structures on
  \(C\boxtimes_{\Rmat} D\) and \(C\boxtimes_{\multunit} D\) are the
  same as well.  Since the isomorphism between these tensor products
  is one of crossed products, it automatically satisfies the
  coherence conditions required for a monoidal functor.
\end{proof}

\section{Braided \texorpdfstring{$C^*$}{C*}-bialgebras and braided
  compact quantum groups}
\label{sec:braid_C_star_bialg}

We are going to define braided \(\Cst\)\nb-bialgebras and use them to
construct ordinary \(\Cst\)\nb-bialgebras by a semidirect product
construction, which is the \(\Cst\)\nb-analogue of what Majid calls
``bosonisation'' in~\cite{Majid:Hopfalg_in_BrdCat}.  We check that
the semidirect product
\(\Cst\)\nb-bialgebra is bisimplifiable if and only if the braided
\(\Cst\)\nb-bialgebra is bisimplifiable.  Thus we may construct
compact quantum groups from two pieces: an ordinary compact quantum
group and a braided quantum group over its codouble.

\begin{definition}
  \label{def:Braid_bialg}
  A \emph{braided \(\Cst\)\nb-bialgebra} over a quasitriangular
  quantum group \(\G=(A,\Comult[A],\Rmat)\) is a
  \(\G\)\nb-\(\Cst\)-algebra \((B,\beta)\) with a
  \(\G\)\nb-equivariant morphism \(\Comult[B]\colon B\to
  B\boxtimes_{\Rmat}B\) which is coassociative:
  \begin{equation}
    \label{eq:braid_coasso}
    (\Comult[B]\boxtimes_{\Rmat}\Id_B)\circ\Comult[B]
    =(\Id_B\boxtimes_{\Rmat}\Comult[B])\circ\Comult[B].
  \end{equation}
  We call~\(\Bialg{B}\) \emph{bisimplifiable} if it satisfies the
  braided Podle\'s conditions
  \begin{equation}
    \label{eq:braided_Podles_cond}
    \Comult[B](B)\cdot \iota_1(B)
    = B\boxtimes_{\Rmat}B
    =\Comult[B](B)\cdot \iota_2(B),
  \end{equation}
  where \(\iota_1\) and~\(\iota_2\) denote the two canonical maps
  \(B\rightrightarrows B\boxtimes_{\Rmat}B\).

  A \emph{braided compact quantum group} over~\(\G\) is a unital,
  bisimplifiable braided \(\Cst\)\nb-bialgebra~\(\Bialg{B}\)
  over~\(\G\).
\end{definition}

In the following, we let \(\G=(A,\Comult[A])\) be \emph{any}
\(\Cst\)\nb-quantum group, and we let~\((B,\Comult[B])\) be a
braided \(\Cst\)\nb-bialgebra over the codouble~\(\Codouble{\G}\)
with its canonical \Rmat\nb-matrix.  Equivalently, \(B\)~is a
\(\Cst\)\nb-bialgebra in the category of \(\G\)\nb-Yetter--Drinfeld
\(\Cst\)\nb-algebras (see
Theorem~\ref{the:Cat_equiv_YD_and_codoub}).  Thus we do not
assume~\(A\) to be quasitriangular any more.  Since we may embed the
coaction category of a quasitriangular \(\Cst\)\nb-quantum group
into the one for its codouble by
Theorem~\ref{the:Codoub_quasi_triag}, our new setting is more
general than the one in Definition~\ref{def:Braid_bialg}.

The monoidal structure on \(\G\)\nb-Yetter--Drinfeld algebras is
given by the tensor product \(C\boxtimes_{\multunit} D\) for the
bicharacter~\(\multunit\) by
Theorem~\ref{the:Cat_equiv_YD_and_codoub}.  So the underlying
\(\Cst\)\nb-algebra only uses the coaction of~\(A\) on~\(C\) and the
coaction of~\(\hat{A}\) on~\(D\).  Both coactions are used to equip
\(C\boxtimes_{\multunit} D\) with a Yetter--Drinfeld algebra
structure, which we need to form tensor products of more than two
factors.  We abbreviate \(\boxtimes=\boxtimes_{\multunit}\).

The \(\Cst\)\nb-algebra~\(A\) carries the canonical continuous
coaction~\(\Comult[A]\) of~\(A\) and a canonical coaction
of~\(\hat{A}\) by \(\Ad(\Dumultunit)\colon a\mapsto
\Dumultunit(a\otimes 1_{\hat{A}})\Dumultunit[*]\).  These two
coactions satisfy the Yetter--Drinfeld compatibility condition.  The
Podle\'s condition for the \(\hat{A}\)\nb-coaction on~\(A\) is not
automatic, however: it is a weak form of regularity.  Since we do
not want to impose any regularity condition on~\(\G\), we make sure
that we do not use the coaction of~\(\hat{A}\) on~\(A\) in the
following constructions.  The \(A\)\nb-coaction~\(\Comult[A]\)
on~\(A\) is enough to define the twisted tensor products
\(A\boxtimes B\) and \(A\boxtimes (B\boxtimes B)\).

\begin{lemma}
  \label{lem:A_coact_AB}
  There are unique coactions of \(A\) and~\(\hat{A}\) on
  \(A\boxtimes B\) and \(A\boxtimes (B\boxtimes B')\) for which the
  canonical embeddings of \(A\), \(B\) and~\(B'\) are equivariant;
  the \(A\)\nb-coactions are continuous, the
  \(\hat{A}\)\nb-coactions are injective, but do not necessarily
  satisfy the Podle\'s condition.  The coactions of \(A\)
  and~\(\hat{A}\) are compatible.  There is a canonical isomorphism
  of triple crossed products \(A\boxtimes (B\boxtimes B') \cong
  (A\boxtimes B)\boxtimes B'\), which is equivariant for the
  coactions of \(A\) and~\(\hat{A}\).
\end{lemma}

\begin{proof}
  If the \(\hat{A}\)\nb-coaction on~\(A\) were continuous, our
  previous theory for coactions of the codouble of~\(\G\) would give
  all the statements immediately.

  Let \(\Multunit\in\U(\Hils\otimes\Hils)\) be a manageable
  multiplicative unitary generating~\((A,\Comult[A])\).  Let
  \(\pi\colon A\to \Bound(\Hils)\) and \(\hat\pi\colon
  \hat{A}\to\Bound(\Hils)\) be the resulting representations.  The
  unitaries \((\hat\pi\otimes\Id_A)(\multunit)\in
  \U(\Comp(\Hils)\otimes A)\) and
  \((\pi\otimes\Id_{\hat{A}})(\Dumultunit)\in\U(\Comp(\Hils)\otimes
  \hat{A})\) are corepresentations because
  \(\multunit\in\U(\hat{A}\otimes A)\) and \(\Dumultunit\in\U(A\otimes
  \hat{A})\) are bicharacters.  These two corepresentations together
  with~\(\pi\) form a faithful covariant representation of
  \((A,\Comult[A],\Ad(\Dumultunit))\).  Moreover, the
  corepresentations \((\hat\pi\otimes\Id_A)(\multunit)\) and
  \((\pi\otimes\Id_{\hat{A}})(\Dumultunit)\) satisfy the
  Yetter--Drinfeld compatibility condition, so they give a
  corepresentation of the codouble~\(\Codouble{\G}\).

  Let \(\beta\colon B\to B\otimes A\) and \(\hat\beta\colon B\to
  B\otimes\hat{A}\) denote the coactions of \(A\) and~\(\hat{A}\)
  on~\(B\) that give the Yetter--Drinfeld algebra structure
  on~\(B\).  We may choose a faithful covariant representation
  \((\rho,U,V)\) of \((B,\beta,\hat\beta)\) on some Hilbert
  space~\(\Hils[K]\).  Thus \(U\) and~\(V\) satisfy the
  Yetter--Drinfeld compatibility condition, so they give a
  corepresentation of the codouble~\(\Codouble{\G}\).

  Now we represent \(A\boxtimes B\) faithfully on
  \(\Hils\otimes\Hils[K]\).  This gives a \(\Cst\)\nb-algebra even
  if~\(\Ad(\Dumultunit)\) is not continuous because the construction
  of \(A\boxtimes B=A\boxtimes_{\multunit} B\) only uses the
  \(A\)\nb-coaction~\(\Comult[A]\) on~\(A\) and the
  \(\hat{A}\)\nb-coaction~\(\hat\beta\) on~\(B\).  The codouble
  of~\(\G\) acts on \(\Hils\otimes\Hils[K]\) by the usual tensor
  product corepresentation.  As in the proof of
  Proposition~\ref{pro:braide_tens_presv_cov_corep}, we get a unique
  coaction of~\(\Codouble{\G}\) on~\(A\boxtimes B\) for which its
  representation on~\(\Hils\otimes\Hils[K]\) and the embeddings of
  \(A\) and~\(B\) are equivariant.  Only the proof of the Podle\'s
  condition breaks down because we do not know the Podle\'s
  condition for the \(\Codouble{\G}\)-coaction on~\(A\).  We may,
  however, split the \(\Codouble{\G}\)-coaction into compatible
  coactions of \(A\) and~\(\hat{A}\), and prove the Podle\'s
  condition for the coaction of~\(A\), just as in the proof of
  Proposition~\ref{pro:braide_tens_presv_cov_corep}, using only the
  \(A\)\nb-equivariance of the embeddings of \(A\) and~\(B\)
  into~\(A\boxtimes B\) and the Podle\'s conditions for the
  \(A\)\nb-coactions on \(A\) and~\(B\).

  Similarly, the construction of the associator \((A\boxtimes
  B)\boxtimes B' \cong A\boxtimes (B\boxtimes B')\) in the proof of
  Theorem~\ref{the:Braided_coaction_Cat} still works, using the
  covariant representation of \(A\boxtimes B\) just constructed, and
  gives the remaining statements.
\end{proof}

Our goal is to construct a coassociative comultiplication on \(C\defeq
A\boxtimes B\) from a braided comultiplication \(\Comult[B]\colon B\to
B\boxtimes B\).  The first ingredient is the morphism
\[
\Id_A\boxtimes \Comult[B]\colon A\boxtimes B \to
A\boxtimes (B\boxtimes B),
\]
which is the unique one with \((\Id_A\boxtimes \Comult[B])\circ
\iota_A=\iota_A\) and \((\Id_A\boxtimes \Comult[B])\circ
\iota_B=\iota_{B\boxtimes B}\circ \Comult[B]\).  Next we construct a
canonical map
\[
\Psi\colon A\boxtimes B\boxtimes B \to
(A\boxtimes B)\otimes (A\boxtimes B).
\]
Under regularity assumptions on~\(\G\), we could construct this by
composing the canonical morphism
\[
j_{124}\colon A\boxtimes B\boxtimes B \to
A\boxtimes B\boxtimes A\boxtimes B
\]
with an isomorphism \(A\boxtimes B\boxtimes A\boxtimes B \cong
(A\boxtimes B)\otimes (A\boxtimes B)\), which exists because the
\(\hat{A}\)\nb-coaction on~\(A\) is inner (see
\cite{Meyer-Roy-Woronowicz:Twisted_tensor}*{Corollary 5.16}).  The
following proposition constructs~\(\Psi\) directly without regularity
assumptions on~\(\G\):

\begin{proposition}
  \label{prop:isom_crossed_and_tensor_prod}
  Let \(B\) and~\(B'\) be \(\G\)\nb-Yetter--Drinfeld algebras, let
  \(\beta\colon B\to B\otimes A\) be the \(A\)\nb-coaction.  There
  is a unique injective morphism
  \[
  \Psi\colon A\boxtimes B\boxtimes B' \to (A\boxtimes
  B)\otimes(A\boxtimes B')
  \]
  that satisfies, for \(a\in A\), \(b\in B\), \(b'\in B'\),
  \begin{equation}
    \label{eq:isomorph_def}
    \begin{aligned}
      \Psi\circ \iota_A(a)&=(\iota_A\otimes \iota_A)\circ\Comult[A](a),\\
      \Psi\circ \iota_B(b)&=(\iota_B\otimes \iota_A)\circ\beta(b),\\
      \Psi\circ \iota_{B'}(b')&=1_{A\boxtimes B}\otimes \iota_{B'}(b').
    \end{aligned}
  \end{equation}
\end{proposition}

Before we prove this technical result, we state our main result and
give a simple example.  Another example is the construction of
quantum U(2) groups from braided quantum SU(2) groups
in~\cite{Kasprzak-Meyer-Roy-Woronowicz:Braided_SU2} (the conventions
in~\cite{Kasprzak-Meyer-Roy-Woronowicz:Braided_SU2} are, however,
slightly different).

\begin{theorem}
  \label{the:bialg_frm_brd}
  Let \(C\defeq A\boxtimes B\) and \(\Comult[C]\defeq \Psi\circ
  (\Id_A\boxtimes \Comult[B])\colon C\to C\otimes C\).  Then
  \(\Bialg{C}\) is a bisimplifiable \(\Cst\)\nb-bialgebra
  whenever~\(\Bialg{B}\) is a bisimplifiable braided
  \(\Cst\)\nb-bialgebra over~\(\G\).  If~\(\Comult[B]\) is
  injective, then so is~\(\Comult[C]\), and vice versa.
\end{theorem}

\begin{corollary}
  \label{cor:CQG_from_BCQG}
  \(\Bialg{C}\) is a compact quantum group if~\(\G\) is a compact
  quantum group and \(\Bialg{B}\) is a braided compact quantum group
  over~\(\G\).
\end{corollary}

\begin{proof}
  The \(\Cst\)\nb-algebra~\(C\) is unital if and only if \(A\)
  and~\(B\) are both unital.  In the unital (compact) case, the
  Podle\'s conditions suffice to characterise compact quantum
  groups.
\end{proof}

\begin{example}
  \label{exa:partial_dual}
  The following example is inspired by the construction of partial
  duals in~\cite{Barvels-Lentner-Schweigert:Partial_dual} in the
  setting of Hopf algebras.  Let~\(K\) be a compact group,
  let~\(\Gamma\) be a discrete group, and let \(\varphi\colon
  \Gamma\to \Aut(K)\) be a group homomorphism.  Let~\(\Gamma\) act on
  \(B\defeq \Cont(K)\) by \(\varphi^*_g f(k) \defeq
  f(\varphi^{-1}_g(k))\) for all \(k\in K\), \(g\in\Gamma\),
  \(f\in\Cont(K)\).  Equip~\(\Cont(K)\) with the trivial coaction
  of~\(\Gamma\).  Since the \(\Gamma\)\nb-coaction on~\(\Cont(K)\)
  is trivial,
  \[
  \Cont_0(\Gamma) \boxtimes \Cont(K)
  \cong \Cont_0(\Gamma) \otimes \Cont(K)
  \cong \Cont_0(\Gamma \times K).
  \]
  The comultiplication~\(\Comult[C]\) is the one that is induced by the
  multiplication in the semidirect product group \(\Gamma\ltimes
  K\).

  We may also view \(\Cont(K)\) as a Yetter--Drinfeld algebra over
  \(\Cred(\Gamma)\) instead of~\(\Cont_0(\Gamma)\).  This gives a
  compact quantum group \(\Cred(\Gamma)\boxtimes \Cont(K)\) by
  Corollary~\ref{cor:CQG_from_BCQG}.  Its underlying
  \(\Cst\)\nb-algebra is canonically isomorphic to the reduced crossed
  product \(\Gamma\ltimes\Cont(K)\) (see
  \cite{Meyer-Roy-Woronowicz:Twisted_tensor}*{Section 6.3}).  Thus the
  unital \(\Cst\)\nb-algebra \(\Gamma\ltimes\Cont(K)\) becomes a
  compact quantum group by Corollary~\ref{cor:CQG_from_BCQG}.  This is
  the partial dual of the group \(\Gamma\ltimes K\), where we dualise
  \(\Cont_0(\Gamma)\) to~\(\Cred(\Gamma)\) and leave~\(\Cont(K)\)
  unchanged.  This example is also a special case of a bicrossed
  product (see~\cite{Baaj-Skandalis:Unitaires}*{Proposition 8.22}).
\end{example}

In the remainder of this section, we will prove
Proposition~\ref{prop:isom_crossed_and_tensor_prod} and
Theorem~\ref{the:bialg_frm_brd}.

First we name the coactions on our Yetter--Drinfeld
algebras: call them \(\beta\colon B\to B\otimes A\),
\(\hat{\beta}\colon B\to B\otimes \hat{A}\), \(\beta'\colon B'\to
B'\otimes A\), \(\hat{\beta}'\colon B'\to B'\otimes \hat{A}\).  We
choose faithful covariant, \(\Codouble{\G}\)-compatible representations
\((\varphi,\corep{U},\corep{V})\) and
\((\varphi',\corep{U}',\corep{V}')\) of \((B,\beta,\hat{\beta})\)
and~\((B',\beta',\hat{\beta}')\) on some Hilbert spaces \(\Hils[L]\)
and~\(\Hils[L']\).  We choose a manageable multiplicative unitary
\(\Multunit
\in \U(\Hils\otimes\Hils)\) generating~\(\G\); it induces
representations \(\pi\colon A\to\Bound(\Hils)\) and \(\hat\pi\colon
\hat{A}\to\Bound(\Hils)\) with
\(\Multunit=(\hat{\pi}\otimes\pi)\multunit\).  Then \(\pi\) and the
corepresentation \((\hat{\pi}\otimes\Id_A)W \in
\U(\Comp(\Hils)\otimes A)\) form a covariant representation of
\((A,\Comult[A])\).  We use these covariant representations of
\(A\), \(B\) and~\(B'\) to realise \(A\boxtimes B\boxtimes B'\) as a
\(\Cst\)\nb-subalgebra of \(\Bound(\Hils\otimes \Hils[L]\otimes
\Hils[L]')\); this gives
\[
A\boxtimes B \boxtimes B'
= \iota_A(A)\cdot \iota_B(B)\cdot \iota_{B'}(B')
\]
for three representations \(\iota_A\), \(\iota_B\), \(\iota_{B'}\) of
\(A\), \(B\) and~\(B'\) on \(\Hils\otimes \Hils[L]\otimes \Hils[L]'\).
We describe these representations as in the proof of
Theorem~\ref{the:Braided_coaction_Cat}.

The representation~\(\iota_A\) is most easy:
\[
\iota_A(a) = \pi(a)\otimes 1_{\Hils[L]\otimes\Hils[L]'}.
\]
To describe~\(\iota_B\), we must represent the universal \(R\)-matrix,
which is essentially~\(\multunit{}^\univ\), on the Hilbert
space~\(\Hils\otimes\Hils[L]\).  The bijection between
corepresentations of \(A\) and representations of~\(\hat{A}^\univ\)
maps the left corepresentation \(\multunit_{1\pi}\in
\U(\hat{A}\otimes\Comp(\Hils))\) to the representation \(\hat\pi\colon
\hat{A}^\univ \to\hat{A}\to\Bound(\Hils)\).  The bijection between
corepresentations of~\(\hat{A}\) and representations of~\(A^\univ\)
maps the right corepresentation \(\corep{V}\in
\U(\Comp(\Hils[L])\otimes \hat{A})\) to the unique representation
\(\rho\colon A^\univ \to \Bound(\Hils[L])\) with \(\multunit_{1\rho} =
\Ducorep{V}\defeq \flip(\corep{V})^*\in
\U(\hat{A}\otimes\Comp(\Hils[L]))\).  Hence the resulting
representation \(\hat\pi\otimes\rho\colon \hat{A}^\univ\otimes
A^\univ\to \Bound(\Hils\otimes\Hils[L])\) maps~\(\multunit{}^\univ\)
to \(\DuCorep{V} \defeq \Ducorep{V}_{\hat\pi 2}\).  Thus the braiding
unitary \(\Braiding{\Hils[L]}{\Hils}\) is \(\DuCorep{V}^*\Sigma\) and
\[
\iota_B(b) = \DuCorep{V}^*_{12}
(1_{\Hils}\otimes\varphi(b)\otimes 1_{\Hils[L]}) \DuCorep{V}_{12}.
\]

The representations \(\iota_A\) and~\(\iota_B\) without the trivial
third leg in~\(\Hils[L]'\) also realise \(A\boxtimes B\) in
\(\Bound(\Hils\otimes\Hils[L])\).  Similarly, we realise \(A\boxtimes
B'\) in \(\Bound(\Hils\otimes\Hils[L]')\) using \(\DuCorep{V}'\defeq
(\hat{\pi}\otimes\Id_{\Hils[L]'})\Ducorep{V}{}' \in
\U(\Hils\otimes\Hils[L]')\) instead of~\(\DuCorep{V}\).

The braiding for \(\Hils[L]\) and~\(\Hils[L]'\) involves a
unitary~\(Z\) given by~\eqref{eq:Z_appendix}.  This unitary may also
be characterised uniquely by the condition
\begin{equation}
  \label{eq:braid_op_U_hatV}
  \Corep{U}_{12}(\DuCorep{V}')^*_{23}Z_{13}
  = (\DuCorep{V}')^*_{23}\Corep{U}_{12}
  \qquad\text{in }\U(\Hils[L]\otimes\Hils\otimes\Hils[L]'),
\end{equation}
where \(\Corep{U}\defeq(\Id_{\Hils[L]}\otimes\pi)\corep{U}\in
\U(\Hils[L]\otimes\Hils)\);
compare~\eqref{eq:commutation_between_corep_of_two_quantum_groups}.
As in the proof of Theorem~\ref{the:Braided_coaction_Cat}, we see
that the representation~\(\iota_{B'}\) is
\[
\iota_{B'}(b') = Z_{23}(\DuCorep{V}')^*_{13}
(1_{\Hils\otimes\Hils[L]}\otimes\varphi'(b'))
(\DuCorep{V}')_{13}Z^*_{23}.
\]

\begin{proof}[Proof of
  Proposition~\textup{\ref{prop:isom_crossed_and_tensor_prod}}]
  Define
  \[
  \Psi(x) \defeq \Multunit_{13} \Corep{U}_{23} (\DuCorep{V}')^*_{34}
  x_{124} \DuCorep{V}'_{34} \Corep{U}^*_{23}
  \Multunit^*_{13}
  \qquad\text{for }x\in \Bound(\Hils\otimes\Hils[L]\otimes\Hils[L]').
  \]
  This is an injective \Star{}homomorphism from
  \(\Bound(\Hils\otimes\Hils[L]\otimes\Hils[L]')\) to
  \(\Bound(\Hils\otimes\Hils[L]\otimes\Hils\otimes\Hils[L]')\).  We
  compute \(\Psi\circ\iota_A\), \(\Psi\circ\iota_B\),
  \(\Psi\circ\iota_{B'}\); this will show that~\(\Psi\) maps
  \(A\boxtimes B\boxtimes B'\) into
  \[
  (A\boxtimes B) \otimes (A\boxtimes B') \subseteq
  \Bound(\Hils\otimes\Hils[L]\otimes\Hils\otimes\Hils[L]')
  \]
  and has the expected values on \(\iota_A(A)\), \(\iota_B(B)\)
  and~\(\iota_{B'}(B')\).

  Since \(\iota_A(a)=\pi(a)_1\), we get \(\Psi\circ \iota_A(a) =
  \Multunit_{13}\pi(a)_1\Multunit_{13}^* = (\iota_A\otimes
  \iota_A)\circ \Comult[A](a)\).  Next, \(\Psi\circ\iota_B(b) =
  \Multunit_{13} \Corep{U}_{23} \DuCorep{V}^*_{12} \varphi(b)_2
  \DuCorep{V}_{12} \Corep{U}^*_{23} \Multunit^*_{13}\).  The
  Yetter--Drinfeld compatibility condition for \(\corep{U}\)
  and~\(\corep{V}\) is equivalent to
  \(\Multunit_{13}\Corep{U}_{23}\DuCorep{V}^*_{12} =
  \DuCorep{V}^*_{12}\Corep{U}_{23}\Multunit_{13}\) in
  \(\U(\Hils\otimes\Hils[L]\otimes\Hils)\).  Using this and the
  covariance condition for \((\varphi,\corep{U})\) with respect
  to~\(\beta\), we compute
  \[
  \Psi\circ \iota_B(b) =
  \DuCorep{V}^*_{12} \Corep{U}_{23} \varphi(b)_2 \Corep{U}^*_{23}
  \DuCorep{V}_{12}
  = (\iota_B\otimes \iota_A)\circ\beta(b).
  \]
  Finally, we compute
  \begin{align*}
    \Psi\circ \iota_{B'}(b')
    &= \Multunit_{13} \Corep{U}_{23} (\DuCorep{V}')^*_{34} Z_{24}
    (\DuCorep{V}')^*_{14} \varphi'(b')_4
    \DuCorep{V}'_{14} Z_{24}^* \DuCorep{V}'_{34} \Corep{U}^*_{23}
    \Multunit^*_{13}\\
    &= \Multunit_{13} (\DuCorep{V}')^*_{34} \Corep{U}_{23}
    (\DuCorep{V}')^*_{14} \varphi'(b')_4 \DuCorep{V}'_{14}
    \Corep{U}^*_{23} \DuCorep{V}'_{34} \Multunit^*_{13}\\
    &= \Multunit_{13} (\DuCorep{V}')^*_{34} (\DuCorep{V}')^*_{14}
    \varphi'(b')_4
    \DuCorep{V}'_{14} \DuCorep{V}'_{34} \Multunit^*_{13}\\
    &= (\DuCorep{V}')^*_{34} \Multunit_{13} \varphi'(b')_4
    \Multunit^*_{13} \DuCorep{V}'_{34}\\
    &= (\DuCorep{V}')^*_{34} \varphi'(b')_4 \DuCorep{V}'_{34}
    = 1_{A\boxtimes B}\otimes \iota_{B'}(b').
  \end{align*}
  The first equality is trivial; the second equality uses
  \eqref{eq:braid_op_U_hatV}; the third equality commutes
  \(\Corep{U}_{23}\) with \(\DuCorep{V}_{14}\) and~\(\varphi'(b)_4\);
  the
  fourth equality uses that~\(\corep{V}'\) is a corepresentation
  of~\(\hat{A}\); the fifth equality commutes \(\Multunit_{13}\)
  and~\(\varphi'(b')_4\); and the last equality is the definition of
  the embedding of \(A\boxtimes B'\) in
  \(\Bound(\Hils\otimes\Hils[L]')\).
\end{proof}

\begin{proof}[Proof of Theorem~\textup{\ref{the:bialg_frm_brd}}]
  Let \(C\defeq A\boxtimes B\) and \(\Comult[C]\defeq
  \Psi\circ(\Id_A\boxtimes\Comult[B])\colon C\to C\otimes C\), where
  we use~\(\Psi\) from
  Proposition~\ref{prop:isom_crossed_and_tensor_prod} in the special
  case \(B=B'\).  We first check that this comultiplication is
  coassociative.  It suffices to check
  \((\Comult[C]\otimes\Id_C)\Comult[C]\circ \iota_A =
  (\Id_C\otimes\Comult[C])\Comult[C]\circ \iota_A\) and
  \((\Comult[C]\otimes\Id_C)\Comult[C]\circ \iota_B =
  (\Id_C\otimes\Comult[C])\Comult[C]\circ \iota_B\).  The first
  statement holds because on elements of the form \(\iota_A(a)\) with
  \(a\in A\), we get \((\Id_A\otimes\Comult[A])\Comult[A](a)\) and
  \((\Comult[A]\otimes\Id_A)\Comult[A](a)\), respectively, embedded
  into \(C\otimes C\otimes C\) via \(\iota_A\otimes \iota_A\otimes
  \iota_A\).

  To check the formula on~\(B\), we also need the two maps
  \begin{align*}
    \Psi'&\colon A\boxtimes B\boxtimes (B\boxtimes B)
    \to (A\boxtimes B)\otimes (A\boxtimes B\boxtimes B),\\
    \Psi''&\colon A\boxtimes (B\boxtimes B)\boxtimes B
    \to (A\boxtimes B\boxtimes B)\otimes (A\boxtimes B)
  \end{align*}
  that we get from Proposition~\ref{prop:isom_crossed_and_tensor_prod}
  for \(B,B\boxtimes B\) and \(B\boxtimes B,B\), respectively.  These
  satisfy, among others,
  \begin{alignat*}{3}
    \Psi'(b_2) &= \beta(b)_{23},&\quad
    \Psi'(b_3) &= b_4,&\quad
    \Psi'(b_4) &= b_5,\\
    \Psi''(b_2) &= \beta(b)_{24},&\quad
    \Psi''(b_3) &= \beta(b)_{34},&\quad
    \Psi''(b_4) &= b_5.
  \end{alignat*}
  Here we use leg numbering notation to distinguish the different
  copies of~\(B\) more clearly.  For instance, \(\beta(b)_{23}\) means
  \((\iota_B\otimes\iota_A)\beta(b)\).  With these maps, we may write
  \begin{align*}
    (\Id_C\otimes(\Id_A\boxtimes\Comult[B]))\circ\Psi|_{B\boxtimes B} &=
    \Psi' \circ (\Id_B\boxtimes\Comult[B]),\\
    %\colon B\boxtimes B
    %\to (A\boxtimes B)\otimes (A\boxtimes B\boxtimes B),\\
    ((\Id_A\boxtimes\Comult[B])\otimes\Id_C)\circ\Psi|_{B\boxtimes B} &=
    \Psi'' \circ (\Comult[B]\boxtimes\Id_B).
    %\colon B\boxtimes B
    %\to (A\boxtimes B\boxtimes B)\otimes (A\boxtimes B).
  \end{align*}
  The second formula uses that~\(\Comult[B]\) is \(\G\)\nb-equivariant
  with respect to the actions \(\beta\) and~\(\beta\bowtie\beta\) and
  that \(\Psi''\) on \(\iota_2(B)\iota_3(B)\)
  is~\(\beta\bowtie\beta\).  Since~\(\beta\) is a coaction, we get
  \[
  (\Id_C\otimes\Psi)\circ\Psi'|_{B\boxtimes B\boxtimes B} =
  (\Psi\otimes\Id_C)\circ\Psi''|_{B\boxtimes B\boxtimes B}.
  \]
  This and the coassociativity of~\(\Comult[B]\) imply
  that~\(\Comult[C]\) is coassociative also on~\(B\).

  Now we turn to the Podle\'s conditions.  We have \(A\boxtimes B =
  \iota_A(A)\iota_B(B) = \iota_B(B)\iota_A(A)\), \(\beta(B)\cdot
  (1\otimes A) = B\otimes A\) because~\(\beta\) satisfies the Podle\'s
  condition, and \(\Comult[A](A)(1\otimes A) = \Comult[A](A)(A\otimes
  1) = A\otimes A\) and \(\Comult[B](B) \cdot B_2 = \Comult[B](B) B_1
  = B\boxtimes B\) because \(A\) and~\(B\) are bisimplifiable.  Thus
  \begin{align*}
    \Comult[C](C)\cdot (1\otimes C)
    &= \Comult[C](\iota_B(B)) \cdot
    \Comult[C](\iota_A(A)) \cdot (1\otimes \iota_A(A))\cdot
    (1\otimes \iota_B(B))
    \\&= \Comult[C](\iota_B(B)) \cdot
    (\iota_A\otimes\iota_A)(\Comult[A](A) \cdot (1\otimes A))
    \cdot (1\otimes \iota_B(B))
    \\&= \Comult[C](\iota_B(B)) \cdot (\iota_A(A)\otimes C)
    = \Psi(\Comult[B](B)_{23}) \cdot (\iota_A(A)\otimes C).
  \end{align*}
  Since \(\Psi\circ\iota_{B'}(b')=1\otimes \iota_{B'}(b')\) is a
  multiplier of \(\iota_A(A)\otimes C\), we may rewrite this as
  \begin{multline*}
    \Psi(\Comult[B](B)_{23}\cdot B_3)
    \cdot (\iota_A(A)\otimes C)
    \\= \Psi((B\boxtimes B)_{23}) \cdot (\iota_A(A)\otimes C)
    = \Psi(B_2) \cdot (\iota_A(A)\otimes C),
  \end{multline*}
  using that~\(\Bialg{B}\) is bisimplifiable.  Finally, the formula
  for~\(\Psi\) on the second leg and the Podle\'s condition
  for~\(\beta\) show that this is \(C\otimes C\), as desired.

  The other Podle\'s condition is proved similarly.  Since~\(\Psi\)
  is injective, \(\Comult[C]\) is injective if and only if
  \(\Id_A\boxtimes \Comult[B]\) is injective.  This is equivalent
  to~\(\Comult[B]\) being injective by
  \cite{Meyer-Roy-Woronowicz:Homomorphisms}*{Proposition 5.6}.
\end{proof}

\appendix

\section{Preliminaries}
\label{sec:preliminaries}

\subsection{Multiplicative unitaries and quantum groups}
\label{sec:multunit_quantum_groups}

\begin{definition}[\cite{Baaj-Skandalis:Unitaires}]
  \label{def:multunit}
  Let~\(\Hils\) be a Hilbert space.  A unitary
  \(\Multunit\in\U(\Hils\otimes\Hils)\) is \emph{multiplicative} if it
  satisfies the \emph{pentagon equation}
  \begin{equation}
    \label{eq:pentagon}
    \Multunit_{23}\Multunit_{12}
    = \Multunit_{12}\Multunit_{13}\Multunit_{23}
    \qquad
    \text{in }\U(\Hils\otimes\Hils\otimes\Hils).
  \end{equation}
\end{definition}

Technical assumptions such as manageability
(\cite{Woronowicz:Multiplicative_Unitaries_to_Quantum_grp}) are
needed to construct \(\Cst\)\nb-algebras out of a
multiplicative unitary.

\begin{theorem}[\cites{Soltan-Woronowicz:Remark_manageable,
    Soltan-Woronowicz:Multiplicative_unitaries,
    Woronowicz:Multiplicative_Unitaries_to_Quantum_grp}]
  \label{the:Cst_quantum_grp_and_mult_unit}
  Let~\(\Hils\) be a separable Hilbert space and
  \(\Multunit\in\U(\Hils\otimes\Hils)\) a manageable
  multiplicative unitary.  Let
  \begin{alignat}{2}
    \label{eq:first_leg_slice}
    A &\defeq \{(\omega\otimes\Id_{\Hils})\Multunit:
    \omega\in\Bound(\Hils)_*\}^\CLS,\\
    \label{eq:second_leg_slice}
    \hat{A} &\defeq \{(\Id_{\Hils}\otimes\omega)\Multunit:
    \omega\in\Bound(\Hils)_*\}^\CLS.
  \end{alignat}
  \begin{enumerate}
  \item \(A\) and \(\hat{A}\) are separable, nondegenerate
    \(\Cst\)\nb-subalgebras of~\(\Bound(\Hils)\).
  \item \(\Multunit\in\U(\hat{A}\otimes
    A)\subseteq\U(\Hils\otimes\Hils)\).  We write~\(\multunit[A]\)
    for~\(\Multunit\) viewed as a unitary multiplier of
    \(\hat{A}\otimes A\) and call it \emph{reduced bicharacter}.
  \item There is a unique morphism \(\Comult[A]\colon A\to A\otimes
    A\) such that
    \begin{equation}
      \label{eq:W_char_in_second_leg}
      (\Id_{\hat{A}}\otimes \Comult[A])\multunit[A]
      = \multunit[A]_{12}\multunit[A]_{13}
      \qquad \text{in }\U(\hat{A}\otimes A\otimes A);
    \end{equation}
    it is \emph{coassociative} and \emph{bisimplifiable}:
    \begin{gather}
      \label{eq:coassociative}
      (\Comult[A]\otimes\Id_A)\circ\Comult[A]
      = (\Id_A\otimes\Comult[A])\circ\Comult[A],
      \\\label{eq:Podles}
      \Comult[A](A)\cdot(1_A\otimes A)
      = A\otimes A
      = (A\otimes 1_A)\cdot\Comult[A](A).
    \end{gather}
  \end{enumerate}
\end{theorem}

A \emph{\(\Cst\)\nb-quantum group} is a \(\Cst\)\nb-bialgebra
\(\Qgrp{G}{A}\) constructed from a manageable multiplicative unitary.
This class contains the locally compact quantum groups of Kustermans
and Vaes~\cite{Kustermans-Vaes:LCQG}, which are defined by the
existence of left and right Haar weights.

The \emph{dual} multiplicative unitary is \(\DuMultunit\defeq
\Flip\Multunit^*\Flip\in\U(\Hils\otimes\Hils)\), where
\(\Flip(x\otimes y)=y\otimes x\).  It is manageable
if~\(\Multunit\) is.  The \(\Cst\)\nb-quantum group
\(\DuQgrp{G}{A}\) generated by~\(\DuMultunit\) is the \emph{dual}
of~\(\G\).  Its comultiplication is characterised by
\begin{equation}
  \label{eq:dual_Comult_W}
  (\DuComult[A]\otimes \Id_A)\multunit[A]
  = \multunit[A]_{23}\multunit[A]_{13}
  \qquad\text{in }\U(\hat{A}\otimes \hat{A}\otimes A).
\end{equation}

\subsection{Corepresentations}

\begin{definition}
  \label{def:corepresentation}
  A (right) \emph{corepresentation} of~\(\G\) on a Hilbert
  space~\(\Hils\) is a unitary \(\corep{U}\in\U(\Comp(\Hils)\otimes A)\)
  with
  \begin{equation}
    \label{eq:corep_cond}
    (\Id_{\Comp(\Hils)}\otimes\Comult[A])\corep{U} =\corep{U}_{12}\corep{U}_{13}
    \qquad\text{in }\U(\Comp(\Hils)\otimes A\otimes A).
  \end{equation}
\end{definition}

Let \(\corep{U}^1\in\U(\Comp(\Hils_1)\otimes A)\) and
\(\corep{U}^2\in\U(\Comp(\Hils_2)\otimes A)\) be corepresentations
of~\(\G\).  An element \(t\in\Bound(\Hils_1,\Hils_2)\) is called
an \emph{intertwiner} if \((t\otimes 1_A)\corep{U}^1 =
\corep{U}^2(t\otimes 1_A)\).  The set of all intertwiners
between \(\corep{U}^1\) and~\(\corep{U}^2\) is denoted
\(\Hom(\corep{U}^1, \corep{U}^2)\).  This gives corepresentations a
structure of \(\textup{W}^*\)\nb-category (see
\cite{Soltan-Woronowicz:Multiplicative_unitaries}*{Sections 3.1--2}).

The \emph{tensor product} of two corepresentations
\(\corep{U}^{\Hils_1}\) and~\(\corep{U}^{\Hils_2}\) is defined by
\begin{equation}
  \label{eq:tens_corep}
  \corep{U}^{\Hils_1}\tenscorep \corep{U}^{\Hils_2}\defeq
  \corep{U}^{\Hils_1}_{13}\corep{U}^{\Hils_2}_{23}
  \qquad\text{in }\U(\Comp(\Hils_1\otimes\Hils_2)\otimes A).
\end{equation}
Routine computations show the following:
\(\corep{U}^{\Hils_1}\tenscorep \corep{U}^{\Hils_2}\) is a
corepresentation; \(\tenscorep\) is associative; and the trivial
\(1\)\nb-dimensional representation is a tensor unit.  Thus
corepresentations form a monoidal \(\textup{W}^*\)-category, which we
denote by~\(\Corepcat(\G)\); see
\cite{Soltan-Woronowicz:Multiplicative_unitaries}*{Section 3.3} for
more details.

\subsection{Coactions}

\begin{definition}
  \label{def:cont_coaction}
  A \emph{continuous \textup(right\textup) coaction} of~\(\G\) on a
  \(\Cst\)\nb-algebra~\(C\) is a morphism \(\gamma\colon C\to C\otimes
  A\) with the following properties:
  \begin{enumerate}
  \item \(\gamma\) is injective;
  \item \(\gamma\) is a comodule structure, that is,
    \((\Id_C\otimes\Comult[A])\gamma = (\gamma\otimes\Id_A)\gamma\);
  \item \(\gamma\) satisfies the \emph{Podleś condition}
    \(\gamma(C)\cdot(1_C\otimes A)=C\otimes A\).
  \end{enumerate}
  We call \((C,\gamma)\) a \emph{\(\G\)\nb-\(\Cst\)\nb-algebra}.  We
  often drop~\(\gamma\) from our notation.

  A morphism \(f\colon C\to D\) between two
  \(\G\)\nb-\(\Cst\)\nb-algebras \((C,\gamma)\) and \((D,\delta)\) is
  \emph{\(\G\)\nb-\hspace{0pt}equivariant} if \(\delta\circ f =
  (f\otimes\Id_A)\circ\gamma\).  Let \(\Mor^{\G}(C,D)\) be the set of
  \(\G\)\nb-equivariant morphisms from~\(C\) to~\(D\).  Let
  \(\Cstcat(\G)\) be the category with \(\G\)\nb-\(\Cst\)-algebras as
  objects and \(\G\)\nb-equivariant morphisms as arrows.
\end{definition}

\begin{definition}
  \label{def:covariant_corep}
  A \emph{covariant representation} of \((C,\gamma,A)\) on a Hilbert
  space~\(\Hils\) is a pair consisting of a corepresentation
  \(\corep{U}\in\U(\Comp(\Hils)\otimes A)\) and a representation
  \(\varphi\colon C\to\Bound(\Hils)\) that satisfy the covariance
  condition
  \begin{equation}
    \label{eq:covariant_corep}
    (\varphi\otimes\Id_A)\circ \gamma(c) =
    \corep{U}(\varphi(c)\otimes 1_A)\corep{U}^*
    \qquad\text{in }\U(\Comp(\Hils)\otimes A)
  \end{equation}
  for all \(c\in C\).  A covariant representation is called
  \emph{faithful} if~\(\varphi\) is faithful.
\end{definition}

Faithful covariant representations always exist by
\cite{Meyer-Roy-Woronowicz:Twisted_tensor}*{Example~4.5}.

\subsection{Universal quantum groups}
\label{sec:univ_qgr}

The universal quantum group
\[
\G^\univ\defeq (A^\univ,\Comult[A^\univ])
\]
associated to \(\Qgrp{G}{A}\) is introduced
in~\cite{Soltan-Woronowicz:Multiplicative_unitaries}.  By
construction, it comes with a \emph{reducing map} \(\Lambda\colon
A^\univ\to A\) and a universal bicharacter \(\maxcorep\in
\U(\hat{A}\otimes A^\univ)\).  This may also be characterised as the
unique bicharacter in~\(\U(\hat{A}\otimes A^\univ)\) that lifts
\(\multunit[A]\in \U(\hat{A}\otimes A)\) in the sense that
\((\Id_{\hat{A}}\otimes\Lambda)\maxcorep=\multunit[A]\).

Similarly, there are unique bicharacters
\(\dumaxcorep\in\U(\hat{A}^\univ\otimes A)\) and \(\multunit{}^\univ\in
\U(\hat{A}^\univ\otimes A^\univ)\) that lift \(\multunit[A]\in
\U(\hat{A}\otimes A)\); the latter is constructed
in~\cite{Kustermans:LCQG_universal} assuming a Haar measure and
in~\cite{Meyer-Roy-Woronowicz:Homomorphisms} in the more general
setting of manageable multiplicative unitaries.
The universality of \(\dumaxcorep\in\U(\hat{A}^\univ\otimes A)\)
says that for any corepresentation
\(\corep{U}\in\U(\Comp(\Hils)\otimes A)\) of~\(\G\) on a Hilbert
space~\(\Hils\), there is a unique representation
\(\rho\colon\hat{A}^\univ\to\Bound(\Hils)\) with
\begin{equation}
  \label{eq:univ_prop_dumaxcprep}
  (\rho\otimes\Id_A)\dumaxcorep = \corep{U}
  \qquad\text{in }\U(\Comp(\Hils)\otimes A).
\end{equation}

\subsection{Bicharacters as quantum group morphisms}
\label{sec:bicharacters_morphisms}

Let \(\Qgrp{G}{A}\) and \(\Qgrp{H}{B}\) be \(\Cst\)\nb-quantum groups.
Let \(\DuQgrp{G}{A}\) and \(\DuQgrp{H}{B}\) be their duals.

\begin{definition}[\cite{Meyer-Roy-Woronowicz:Homomorphisms}*{Definition
    16}]
  \label{def:bicharacter}
  A \emph{bicharacter from \(\G\) to~\(\DuG[H]\)} is a unitary
  \(\bichar\in\U(\hat{A}\otimes \hat{B})\) with
  \begin{alignat}{2}
    \label{eq:bichar_char_in_first_leg}
    (\DuComult[A]\otimes\Id_{\hat{B}})\bichar
    &=\bichar_{23}\bichar_{13}
    &\qquad &\text{in }
    \U(\hat{A}\otimes\hat{A}\otimes \hat{B}),\\
    \label{eq:bichar_char_in_second_leg}
    (\Id_{\hat{A}}\otimes\DuComult[B])\bichar
    &=\bichar_{12}\bichar_{13}
    &\qquad &\text{in }
    \U(\hat{A}\otimes \hat{B}\otimes \hat{B}).
  \end{alignat}
\end{definition}

Bicharacters in \(\U(\hat{A}\otimes B)\) are interpreted as quantum
group morphisms from~\(\G\) to~\(\G[H]\)
in~\cite{Meyer-Roy-Woronowicz:Homomorphisms}.  We mainly use
bicharacters in \(\U(\hat{A}\otimes \hat{B})\) and rewrite some
definitions in~\cite{Meyer-Roy-Woronowicz:Homomorphisms} in this
setting.

\begin{definition}
  \label{def:right_quantum_morphism}
  A \emph{right quantum group morphism} from~\(\G\) to~\(\DuG[H]\) is
  a morphism \(\Delta_R\colon A\to A\otimes\hat{B}\) such that the
  following diagrams commute:
  \begin{equation}
    \label{eq:right_homomorphism}
    \begin{tikzpicture}[baseline=(current bounding box.west)]
      \matrix(m)[cd,column sep=4em]{
        A&A\otimes \hat{B}\\
        A\otimes A& A\otimes A\otimes \hat{B}\\
      };
      \draw[cdar] (m-1-1) -- node {\(\Delta_R\)} (m-1-2);
      \draw[cdar] (m-1-1) -- node[swap] {\(\Comult[A]\)} (m-2-1);
      \draw[cdar] (m-1-2) -- node[swap] {\(\Comult[A]\otimes\Id_{\hat{B}}\)} (m-2-2);
      \draw[cdar] (m-2-1) -- node[swap] {\(\Id_A\otimes\Delta_R\)} (m-2-2);
    \end{tikzpicture}
    \quad
    \begin{tikzpicture}[baseline=(current bounding box.west)]
      \matrix(m)[cd,column sep=4em]{
        A&A\otimes \hat{B}\\
        A\otimes \hat{B}& A\otimes \hat{B}\otimes \hat{B}\\
      };
      \draw[cdar] (m-1-1) -- node {\(\Delta_R\)} (m-1-2);
      \draw[cdar] (m-1-1) -- node[swap] {\(\Delta_R\)} (m-2-1);
      \draw[cdar] (m-2-1) -- node[swap] {\(\Delta_R\otimes\Id_{\hat{B}}\)} (m-2-2);
      \draw[cdar] (m-1-2) -- node[swap] {\(\Id_A\otimes\DuComult[B]\)} (m-2-2);
    \end{tikzpicture}
  \end{equation}
\end{definition}

The following theorem summarises some of the main results
of~\cite{Meyer-Roy-Woronowicz:Homomorphisms}.

\begin{theorem}
  \label{the:equivalent_notion_of_homomorphisms}
  There are natural bijections between the following sets:
  \begin{enumerate}
  \item bicharacters \(\bichar\in\U(\hat{A}\otimes\hat{B})\)
    from~\(\G\) to~\(\DuG[H]\);
  \item bicharacters \(\Dubichar\in\U(\hat{B}\otimes\hat{A})\)
    from~\(\G[H]\) to~\(\DuG\);
  \item right quantum group homomorphisms \(\Delta_R\colon A\to
    A\otimes \hat{B}\);
  \item functors \(F\colon\Cstcat(\G)\to\Cstcat(\DuG[H])\) with
    \(\Forget_{\DuG[H]}\circ F=\Forget_{\G}\) for the forgetful
    functor \(\Forget_{\G}\colon\Cstcat(\G)\to\Cstcat\);
  \item Hopf \Star{}homomorphisms \(f\colon A^\univ\to\hat{B}^\univ\)
    between universal quantum groups;
  \item bicharacters
    \(\bichar^\univ\in\U(\hat{A}^\univ\otimes\hat{B}^\univ)\).
  \end{enumerate}
  The first bijection maps a bicharacter~\(\bichar\) to
  \begin{equation}
    \label{eq:dual_bicharacter}
    \Dubichar\defeq\flip(\bichar^*).
  \end{equation}
  A bicharacter~\(\bichar\) and a right quantum group
  homomorphism~\(\Delta_R\) determine each other uniquely via
  \begin{equation}
    \label{eq:def_V_via_right_homomorphism}
    (\Id_{\hat{A}} \otimes \Delta_R)(\multunit[A])
    = \multunit[A]_{12}\bichar_{13}.
  \end{equation}
  The functor~\(F\) associated to~\(\Delta_R\) is the unique one that
  maps \((A,\Comult[A])\) to \((A,\Delta_R)\).  In general, \(F\) maps
  a continuous \(\G\)\nb-coaction \(\gamma\colon C\to C\otimes A\) to
  the unique \(\DuG[H]\)\nb-coaction \(\delta\colon C\to C\otimes
  \hat{B}\) for which the following diagram commutes:
  \begin{equation}
    \label{eq:right_quantum_group_homomorphism_as_fucntor}
    \begin{tikzpicture}[baseline=(current bounding box.west)]
      \matrix(m)[cd,column sep=4.5em]{
        C&C\otimes A\\
        C\otimes\hat{B}& C\otimes A\otimes \hat{B}\\
      };
      \draw[cdar] (m-1-1) -- node {\(\gamma\)} (m-1-2);
      \draw[cdar] (m-1-1) -- node[swap] {\(\delta\)} (m-2-1);
      \draw[cdar] (m-1-2) -- node {\(\Id_C\otimes\Delta_R\)} (m-2-2);
      \draw[cdar] (m-2-1) -- node[swap] {\(\gamma\otimes\Id_{\hat{B}}\)} (m-2-2);
    \end{tikzpicture}
  \end{equation}
  The bicharacter in \(\U(\hat{A}\otimes\hat{B})\)
  associated to a Hopf \Star{}homomorphism \(f\colon A^\univ\to
  \hat{B}^\univ\) is \(\bichar \defeq (\Id_{\hat{A}}\otimes
  \Lambda_{\hat{B}}f)(\maxcorep[A])\), where
  \(\maxcorep[A]\in\U(\hat{A}\otimes A^\univ)\) is the
  unique bicharacter lifting \(\multunit[A]\in
  \U(\hat{A}\otimes A)\) and \(\Lambda_{\hat{B}}\colon
  \hat{B}^\univ\to\hat{B}\) is the reducing map.
\end{theorem}

\subsection{Twisted tensor products}

Let \(\gamma\colon C\to C\otimes A\) and \(\delta\colon D\to D\otimes
B\) be coactions of \(\G\) and~\(\G[H]\) on \(\Cst\)\nb-algebras~\(C\)
and~\(D\), respectively.  We are going to describe the twisted tensor
product
\[
C\boxtimes D \defeq (C,\gamma)\boxtimes_{\bichar} (D,\delta)
\]
for a bicharacter \(\bichar\in\U(\hat{A}\otimes\hat{B})\).  Let
\((\varphi,\corep{U}^{\Hils})\) and \((\psi,\corep{U}^{\Hils[K]})\)
be faithful covariant representations of \((C,\gamma)\)
and~\((D,\delta)\) on Hilbert spaces \(\Hils\) and~\(\Hils[K]\),
respectively.  Thus
\(\corep{U}^{\Hils}\in\U(\Comp(\Hils)\otimes A)\) and
\(\corep{U}^{\Hils[K]} \in\U(\Comp(\Hils[K])\otimes B)\) are
corepresentations of \(\G\) and~\(\G[H]\).  Let \(\rho^{\Hils}\colon
\hat{A}^\univ\to\Bound(\Hils)\) and \(\rho^{\Hils[K]}\colon
\hat{B}^\univ\to\Bound(\Hils[K])\) be the corresponding
representations of the universal duals.  Let
\(\bichar^\univ\in\U(\hat{A}\otimes \hat{B})\) lift~\(\bichar\), see
Theorem~\ref{the:equivalent_notion_of_homomorphisms}.  Let
\begin{equation}
  \label{eq:Z_appendix}
  Z\defeq
  (\rho^{\Hils}\otimes\rho^{\Hils[K]})(\bichar^\univ)^* \in
  \U(\Hils\otimes\Hils[K]).
\end{equation}
The proof of
\cite{Meyer-Roy-Woronowicz:Twisted_tensor}*{Theorem~4.1} shows that
this is the unique \(Z\in\U(\Hils\otimes\Hils[K])\) with
\begin{equation}
  \label{eq:commutation_between_corep_of_two_quantum_groups}
  \corep{U}^{\Hils}_{1\alpha}\corep{U}^{\Hils[K]}_{2\beta} Z_{12}
  =\corep{U}^{\Hils[K]}_{2\beta}\corep{U}^{\Hils}_{1\alpha}
  \qquad\text{in }\U(\Hils\otimes\Hils[K]\otimes\Hils[L])
\end{equation}
for any \(\bichar\)\nb-Heisenberg pair~\((\alpha,\beta)\) on any
Hilbert space~\(\Hils[L]\).

Define representations \(\iota_C = \varphi_1\) and \(\iota_D =
\tilde\psi_2\) of \(C\) and~\(D\) on \(\Hils\otimes\Hils[K]\) by
\begin{equation}
  \label{eq:represent_braided_tensor_Hilbert}
  \begin{aligned}
    \iota_C(c) = \varphi_1(c) &\defeq \varphi(c)\otimes 1_{\Hils[K]},\\
    \iota_D(d) = \tilde\psi_2(d) &\defeq Z (1_{\Hils}\otimes \psi(d))Z^*.
  \end{aligned}
\end{equation}

\begin{theorem}[\cite{Meyer-Roy-Woronowicz:Twisted_tensor}*{Lemma
    3.20, Theorem 4.3, Theorem 4.9}]
  \label{the:crossed_prod}
  The subspace
  \[
  C\boxtimes D \defeq
  \varphi_1(C)\cdot\tilde{\psi}_2(D)\subset
  \Bound(\Hils\otimes\Hils[K])
  \]
  is a nondegenerate \(\Cst\)\nb-subalgebra.  The crossed product
  \((C\boxtimes D,\iota_C,\iota_D)\), up to equivalence, does not
  depend on the faithful covariant representations
  \((\corep{U}^{\Hils},\varphi)\) and~\((\corep{U}^{\Hils[K]},\psi)\).
\end{theorem}

We call \(C\boxtimes_{\bichar}D\) the \emph{twisted tensor product} of
\(C\) and~\(D\).  It generalises the minimal tensor product of
\(\Cst\)\nb-algebras.

\subsection{Heisenberg pairs}
\label{subsec:Heis_pair_crossed_prod}

Let \(\Qgrp{G}{A}\) and \(\Qgrp{H}{B}\) be \(\Cst\)\nb-quantum groups
and let \(\multunit[A]\in\U(\hat{A}\otimes A)\) and
\(\multunit[B]\in\U(\hat{B}\otimes B)\) be their reduced bicharacters,
respectively.  Let \(\bichar\in\U(\hat{A}\otimes\hat{B})\) be a
bicharacter from~\(\G\) to~\(\DuG[H]\).

\begin{definition}[\cite{Meyer-Roy-Woronowicz:Twisted_tensor}*{Definition
    3.1}]
  \label{def:V-Heisenberg_pair}
  Let~\(E\) be a~\(\Cst\)\nb-algebra and let \(\alpha\colon A\to E\)
  and \(\beta\colon B\to E\) be morphisms.  The
  pair~\((\alpha,\beta)\) is a \emph{\(\bichar\)\nb-Heisenberg pair}
  or briefly \emph{Heisenberg pair} on~\(E\) if
  \begin{equation}
    \label{eq:V-Heisenberg_pair}
    \multunit[A]_{1\alpha}\multunit[B]_{2\beta}
    = \multunit[B]_{2\beta}\multunit[A]_{1\alpha} \bichar_{12}
    \qquad\text{in }\U(\hat{A}\otimes\hat{B}\otimes E);
  \end{equation}
  here \(\multunit[A]_{1\alpha} \defeq
  ((\Id_{\hat{A}}\otimes\alpha)\multunit[A])_{13}\) and
  \(\multunit[B]_{2\beta} \defeq
  ((\Id_{\hat{B}}\otimes\beta)\multunit[B])_{23}\).
\end{definition}

The following lemma is routine to check:

\begin{lemma}
  \label{lemm:Heis_comult}
  Let~\((\alpha,\beta)\) be a \(\bichar\)\nb-Heisenberg pair on a
  \(\Cst\)\nb-algebra~\(E\).  Define \(\alpha'\colon A\to A\otimes
  B\otimes E\) and \(\beta'\colon B\to A\otimes B\otimes E\) by
  \(\alpha'(a)\defeq
  \bigl((\Id_{A}\otimes\alpha)\Comult[A](a)\bigr)_{13}\) and
  \(\beta'(b)\defeq\bigl((\Id_{B}\otimes\beta)\Comult[B](b)\bigr)_{23}\).
  This is again a \(\bichar\)\nb-Heisenberg pair.
\end{lemma}

Equip \(A\) and~\(B\) with the standard coactions \(\Comult[A]\)
and~\(\Comult[B]\) of \(\G\) and~\(\G[H]\), respectively, and form the
tensor product \(A\boxtimes_{\bichar}B\).  This plays a special role,
as explained after Proposition~5.6
in~\cite{Meyer-Roy-Woronowicz:Twisted_tensor}: coactions
\(\gamma\colon C\to C\otimes A\) and \(\delta\colon D\to D\otimes B\)
induce a canonical map
\[
\gamma\boxtimes \delta\colon C\boxtimes_{\bichar} D
\to (C\otimes A)\boxtimes_{\bichar} (D\otimes B)
\cong C\otimes D \otimes (A\boxtimes_{\bichar} B).
\]

\begin{proposition}
  \label{pro:rep_cros_vs_heis}
  Let~\(\Hils\) be a Hilbert space and let
  \(\varphi\colon A\boxtimes_{\bichar}B \to \Bound(\Hils)\) be a
  representation.  Define \(\alpha\defeq\varphi\circ\iota_A\colon A
  \to \Bound(\Hils)\) and \(\beta\defeq\varphi\circ\iota_B\colon B \to
  \Bound(\Hils)\).  The pair~\((\alpha,\beta)\) is a
  \(\bichar\)\nb-Heisenberg pair.
\end{proposition}

\begin{proof}
  This is the only place where we use the construction of twisted
  tensor products through Heisenberg pairs in
  \cite{Meyer-Roy-Woronowicz:Twisted_tensor}*{Section 3}.  Let
  \((\alpha',\beta')\) be a \(\bichar\)\nb-Heisenberg pair on a
  \(\Cst\)\nb-algebra~\(E\).  Define morphisms
  \begin{alignat*}{2}
   \iota_A\colon A&\to A\otimes B\otimes E,&\qquad
   a\mapsto \bigl((\Id_{A}\otimes\alpha')\Comult[A](a)\bigr)_{13},\\
   \iota_B\colon B&\to A\otimes B\otimes E,&\qquad
   b\mapsto \bigl((\Id_{B}\otimes\beta')\Comult[B](b)\bigr)_{23}.
  \end{alignat*}
  Then \(A\boxtimes_{\bichar} B \cong \iota_A(A)\cdot \iota_B(B)\).
  Lemma~\ref{lemm:Heis_comult} shows that \((\iota_A,\iota_B)\) is a
  \(\bichar\)\nb-Heisenberg pair on \(A\boxtimes_{\bichar}B\).  Hence
  \((\varphi\circ\iota_A, \varphi\circ\iota_B)\) is a
  \(\bichar\)\nb-Heisenberg pair on~\(\Hils\).
\end{proof}

\begin{corollary}
  \label{cor:tensor_determines_bichar}
  If \(A\boxtimes_{\bichar}B \cong A\boxtimes_{\bichar'}B\) as crossed
  products for two bicharacters
  \(\bichar,\bichar'\in\U(\hat{A}\otimes\hat{B})\), then
  \(\bichar=\bichar'\).
\end{corollary}

\begin{proof}
  Let \(\Hils\), \(\varphi\), \(\alpha\), \(\beta\) as in
  Proposition~\ref{pro:rep_cros_vs_heis}.  The pair \((\alpha,\beta)\)
  is both a \(\bichar\)\nb-Heisenberg pair and a
  \(\bichar'\)\nb-Heisenberg pair by
  Proposition~\ref{pro:rep_cros_vs_heis}.  The commutation
  relation~\eqref{eq:V-Heisenberg_pair} that characterises Heisenberg
  pairs gives
  \[
  \multunit_{2\beta}\multunit_{1\alpha}\bichar_{12}
  = \multunit_{1\alpha}\multunit_{1\beta}
  = \multunit_{2\beta}\multunit_{1\alpha}\bichar'_{12}
  \quad\text{in }\U(\hat{A}\otimes\hat{A}\otimes\Comp(\Hils)).
  \]
  Thus \(\bichar=\bichar'\).
\end{proof}

\end{document}